%% file: arxiv_version_2_iwasawa_rigidity.tex
\documentclass[12pt]{amsart}

\input{macros_kmx.tex}

\begin{document}

\begin{abstract}
Let  $N\subset\gl(n,\R)$  be the group of upper triangular matrices with $1$s on the diagonal, equipped with the standard Carnot group structure.  We show that quasiconformal homeomorphisms of $N$, and more generally Sobolev mappings with nondegenerate Pansu differential,   are rigid   when $n\geq 4$; this settles the Regularity Conjecture for such groups.  This result is deduced from a rigidity theorem for the manifold of complete flags in $\R^n$.    Similar results also hold in the complex and quaternionic cases.
\end{abstract}

\title{Rigidity of flag manifolds}
\author{Bruce Kleiner}
\thanks{BK was supported by NSF grants DMS-1711556 and DMS-2005553, and a Simons Collaboration grant.}
\email{bkleiner@cims.nyu.edu}
\address{Courant Institute of Mathematical Science, New York University, 251 Mercer Street, New York, NY 10012}
\author{Stefan M\"uller}
\thanks{SM has been supported by the Deutsche Forschungsgemeinschaft (DFG, German Research Foundation) through
the Hausdorff Center for Mathematics (GZ EXC 59 and 2047/1, Projekt-ID 390685813) and the 
collaborative research centre  {\em The mathematics of emerging effects} (CRC 1060, Projekt-ID 211504053).  This work was initiated during a sabbatical of SM at the Courant Institute and SM would like to thank  R.V. Kohn and the Courant Institute
members and staff for 
their  hospitality and a very inspiring atmosphere.}
\email{stefan.mueller@hcm.uni-bonn.de}
\address{Hausdorff Center for Mathematics, Universit\"at Bonn, Endenicher Allee 60, 53115 Bonn}
\author{Xiangdong Xie}
\thanks{XX has been supported by Simons Foundation grant \#315130.}
\email{xiex@bgsu.edu}
\address{Dept. of Mathematics and Statistics, Bowling Green State University, Bowling Green, OH 43403}

\maketitle

\tableofcontents

\setcounter{section}{0}
\section{Introduction}
\label{sec_intro}

This is part of a series of papers \cite{KMX1,KMX2,kmx_approximation_low_p,kmx_rumin} on geometric mapping theory in Carnot groups, in which we establish regularity, rigidity, and partial rigidity results for bilipschitz, quasiconformal, or more generally, Sobolev mappings, between Carnot groups.  Our focus in this paper is on the Iwasawa $N$ group for $\sl(n,\F)$ and the associated flag manifold, for $\F\in \{\R,\C,\H\}$.

To state our main theorem we first briefly recall a few facts; see Section~\ref{sec_preliminaries} for details and references.  For simplicity we will stick to the case $\F=\R$ in this introduction.  

Fix $n\geq 3$.  
Let $N\subset \gl(n,\R)$ denote the subgroup of upper triangular matrices with $1$s on the diagonal, with Lie algebra $\fn\subset \gll(n,\R)$, and Carnot group structure given by the grading $\fn=\oplus_i V_i$, where $V_i=\{A\in \fn\mid A_{jk}=0\;\text{if}\; k\neq j+i\}$
corresponds to the $i$th superdiagonal.  Let $\f$ be the manifold of complete flags in $\R^n$, i.e. the collection of nested families of linear subspaces
$$
\{0\}\subsetneq W_1\subsetneq\ldots\subsetneq W_{n-1}\subsetneq\R^n\,.
$$
We consider a standard subbundle $W$ of the tangent bundle  $T\f$ which is invariant under both the natural action $\gl(n,\R)\acts \f$ and the diffeomorphism $\psi:\f\ra\f$ induced by orthogonal complement: 
$$
\psi(W_1,\ldots,W_{n-1})=(W_{n-1}^\perp,\ldots,W_1^\perp)\,.
$$
There is a dense open subset $\hat N\subset \f$ and a diffeomorphism $\al:N\ra \hat N$ carrying the horizontal bundle $V_1\subset TN$ to $W\restr_{\hat N}$, i.e.   $(\hat N,W\restr_{\hat N})$ is contact diffeomorphic to $(N,V_1)$.
 
 We equip $(\f,W)$ with a Carnot-Caratheodory distance $d_{CC}$, and let $\nu$ denote the homogenous dimension of $(\f,W)$.  
 
Our main result is a rigidity theorem for Sobolev mappings satisfying a nondegeneracy condition:
\begin{theorem}
\label{thm_main}
Suppose $n\geq 4$.  Let $U\subset \f$ be a connected open subset, and $f:U\ra \f$ be a $W^{1,p}_{\loc}$-mapping for $p>\nu$, such that the Pansu differential is an isomorphism almost everywhere.  Then $f$ is the restriction of a diffeomorphism $\f\ra \f$   of the form $\psi^{\eps}\circ g$ where $g\in \gl(n,\R)$, $\eps\in \{0,1\}$. 
\end{theorem}
In the theorem and below we use the convention that $\psi^0=\id$ and $\psi^1=\psi$.
 Note that the rigidity assertion is false when $n=3$, because $(\f,W)$ is locally equivalent to the Heisenberg group, and hence has an infinite dimensional group of contact diffeomorphisms, which are all Sobolev mappings with nondegenerate Pansu differential.  Quasiconformal homeomorphisms 
are $W^{1,p}_{\loc}$-mappings for some $p>\nu$  \cite{heinonen_koskela}, and therefore we obtain: 
\begin{corollary}
\label{cor_main_cor}
When $n\geq 4$, then any quasiconformal homeomorphism $\f\supset U\ra U'\subset \f$ between connected open subsets is the restriction of a diffeomorphism  of the form $\psi^{\eps}\circ g$ for some $g\in \gl(n,\R)$, $\eps\in \{0,1\}$.
\end{corollary}
Since $(N,V_1)$ is contact diffeomorphic to $(\hat N,W\restr_{\hat N})$, the above results give a classification of Sobolev mappings $N\supset U\ra N$ with nondegenerate Pansu differential; this applies in particular quasiconformal homeomorphisms and quasiregular maps by \cite[Section 5]{kmx_approximation_low_p}.

We also obtain more refined rigidity for global quasiconformal homeomorphisms of $N$:
\begin{theorem}
\label{thm_global_qc_rigidity}
Any quasiconformal homeomorphism $N\ra N$ is affine, i.e. the composition of a graded automorphism and a left translation.
\end{theorem}

We mention that there was earlier work related to the results above, under additional regularity assumptions.   Corollary~\ref{cor_main_cor} was previously known for real analytic diffeomorphisms \cite{tanaka_differential_systems_graded_lie_algebras,yamaguchi_differential_systems}, $C^2$ diffeomorphisms \cite{cowling_et_al_iwasawa_n_groups,contact_conformal_maps_parabolic_geometry}, and Euclidean bilipschitz homeomorphisms \cite{lelmi}.  After the first version of this paper was posted on the arxiv, we learned that Theorem \ref{thm_main} was shown in \cite{cowling_calclulus_nilpotent_groups} for mappings whose Pansu differential satisfies a nonoscillation condition (see below for more discussion).

\bigskip
\subsection*{Background and historical notes.} We give only the briefest indication of some of the context, and the references provided are only a sampling of the literature.  See \cite{KMX1,KMX2}  for more extensive discussion and references.

For a diffeomorphism $\f\supset U\ra U'\subset\f$, the quasiconformality condition reduces (locally) to the condition of being contact, i.e. preserving the subbundle $W\subset T\f$.  The study of contact diffeomorphisms has a long history in differential geometry and exterior differential systems going back to Cartan \cite{cartan_1904,singer_sternberg_infinite_groups_lie_cartan,tanaka_differential_systems_graded_lie_algebras,yamaguchi_differential_systems,contact_conformal_maps_parabolic_geometry,ottazzi_warhurst}.
The particular case considered in this paper -- the flag manifold $\f$ equipped with the subbundle $W$ --   was treated in the papers \cite{cowling_et_al_iwasawa_n_groups,contact_conformal_maps_parabolic_geometry} 
along with other parabolic geometries.   

Quasiconformal homeomorphisms in $\R^n$, $n\geq 3$, have been heavily studied since the 1960s.  In  the Carnot group setting they first appeared in Mostow's rigidity theorem \cite{mostow_strong_rigidity}, and have been investigated intensely since roughly 1990, due to influential contributions of Gromov, Pansu, and applications to geometric group theory.  For almost 30 years, these historical threads have merged with other developments in analysis on metric spaces \cite{heinonen_koskela,cheeger_differentiability} and PDE  \cite{capogna,capogna_cowling,iwaniec_martin_quasiregular_even_dimensions}.  The rigidity (or regularity) problem for quasiconformal mappings of Carnot groups crystallized over time, first appearing implicitly in the seminal paper \cite{pansu}, 
then in \cite[Remark 6.12]{heinonen_calculus_carnot_groups}, \cite{koranyi_math_review}
and finally taking shape with Ottazzi-Warhurst's notion of a rigid Carnot group  \cite[p.2]{ottazzi_warhurst}, \cite{ottazzi_warhurst_algebraic_prolongation}, explicitly upgraded to a conjecture in \cite{KMX1}.   Corollary~\ref{cor_main_cor} settles the Regularity Conjecture for the Carnot group $N$.     Rigidity of quasiconformal mappings is an instance of rigidity for (solutions to) weak differential inclusions, and more generally partial differential relations, about which there is a vast literature, see for instance \cite{reshetnyak_space_mappings_bounded_distortion,iwaniec_martin_quasiregular_even_dimensions,vodopyanov_foundations,nash54, tartar79, murat81, gromov_pdr, scheffer93, dacorogna_marcellini99, muller99, muller_sverak03, delellis_szekelyhidi09, delellis_szekelyhidi16, isett18, buckmaster_vicol19}.  

\bigskip\bigskip
\subsection*{Discussion of the proof}~ 
For every $1\leq j\leq n-1$, let $\pi_j:\f\ra G(j,n)$ be the $\gl(n,\R)$-equivariant fibration given by $\pi_j(W_1,\ldots,W_{n-1})=W_j$; here $G(j,n)$ is the Grassmannian of $j$-planes in $\R^n$.

The first step of the proof of Theorem~\ref{thm_main}, which is implemented in Sections~\ref{sec_preliminaries} and \ref{sec_sobolev_mappings_preservation_foliations}, is to show that the mapping $f$ preserves the fibers of the fibration $\pi_j$ for all $1\leq j\leq n-1$ (possibly after first post-composing $f$ with the orthogonal complement mapping $\psi:\f\ra \f$).  The heart of the argument is to exclude a certain type of oscillatory behavior in the Pansu differential, and this is achieved by using the Pullback Theorem from \cite{KMX1}.  

The second step of the proof (Section~\ref{sec_rigidity_foliation_preserving_maps}) is to show that a continuous map $U\ra\f$ which preserves the fibrations $\pi_j$, and satisfies a certain nondegeneracy condition, must agree with the diffeomorphism $\f\ra \f$ induced by some element $g\in \gl(n,\R)$.  Our approach to this was inspired by incidence geometry -- the Fundamental Theorem of Projective Geometry and its generalization by Tits.
See Section~\ref{subsec_fibration_preserving_n_3_case} for a sketch of this argument in the $n=3$ case.   

We now compare our argument with previous work.

The paper \cite{cowling_calclulus_nilpotent_groups} has a version of Theorem~\ref{thm_main} with the additional assumption that oscillatory behavior of the Pansu differential -- the key point addressed by the first step in our proof -- does not occur.  The papers \cite{cowling_calclulus_nilpotent_groups,cowling_cap_et_al_heisenberg_sl3r_rigidity,mccallum_local_to_global} all contain rigidity theorems which, like the second step in our proof, were inspired by classical incidence geometry; however, none of these results covers the cases needed for Theorem~\ref{thm_main}.

The approaches to Corollary~\ref{cor_main_cor} taken in \cite{tanaka_differential_systems_graded_lie_algebras,yamaguchi_differential_systems,cowling_et_al_iwasawa_n_groups,contact_conformal_maps_parabolic_geometry,lelmi} are all based on classifying the contact vector fields, and then using an ``integration'' argument to obtain rigidity for contact mappings \cite[Theorem 3.1]{kobayashi_transformation_groups}, \cite{palais_global_formulation}.  To carry out the latter step one has to know that the pushforward of a smooth contact vector field under a contact mapping is still a contact vector field (although possibly of low regularity, a priori).  This pushforward assertion is obvious for $C^2$ diffeomorphisms and was shown to hold for Euclidean bilipschitz homeomorphisms in \cite{lelmi}; this appears to be the minimal regularity needed to implement such an argument.

\bigskip
\subsection*{Organization of the paper}~ The proof of Theorem~\ref{thm_main} is carried out in Sections~\ref{sec_sobolev_mappings_preservation_foliations}-\ref{sec_proof_thm_main}.  The statement and proofs of the analogous results in the complex and quaternionic cases are in Section~\ref{sec_complex_quaternionic}.  Rigidity for global quasiconformal homeomorphisms (Theorem~\ref{thm_global_qc_rigidity}) is proven in Section~\ref{sec_global_qc_homeos}. 

\bigskip
\subsection*{Acknowledgements}  The authors would like to thank Michael Cowling for pointing out reference \cite{cowling_calclulus_nilpotent_groups}.

\section{Preliminaries}
\label{sec_preliminaries}

\subsection{The Iwasawa $N$ group, and the manifold of complete flags}~
Most of this subsection is standard material from Lie theory.  We have tried to give a more self-contained treatment accessible to readers less familiar with Lie theory; see for instance \cite{contact_conformal_maps_parabolic_geometry} for another approach.

\subsection*{Notation} Let $n\geq 2$ be an integer.  We will use the following notation: 
\bit
\item $P^+,P^-\subset \gl(n,\R)$ will denote the subgroups of  upper and lower triangular matrices, respectively.
\item $N\subset P^+$ denotes the subgroup of upper triangular matrices with $1$s on the diagonal.  
\item  $G:=\pgl(n,\R)=\gl(n,\R)/\{\la \id\mid\la\neq 0\}$ is projective linear group.  
\item We let $P:=P^-/\{\la\id\mid \la\neq 0\}\subset \gl(n,\R)/\{\la\id\mid \la \neq 0\}=G$ be the image of $P^-\subset \gl(n,\R)$ under quotient map $\gl(n,\R)\ra \pgl(n,\R)$.  
\item Since  $N\cap P^-=\{e\}$ the quotient map $\gl(n,\R)\ra G$ restricts to an embedding $N\hookrightarrow G$, and we identify $N$ with its image in $G$.
\item $\fn$, $\fp^\pm$, $\fp$, $\gll(n,\R)$, and $\fg$ denote the Lie algebras of $N$, $P^\pm$, $P$, $\gl(n,\R)$, and $G$, respectively.
\item $X_{ij}\in \gll(n,\R)$ is the matrix with a $1$ in the $ij$-entry and $0$s elsewhere, so 
\begin{equation}
\label{eqn_bracket_relations}
\begin{aligned} 
X_{i_1,j_1}X_{i_2,j_2}&=\de_{i_2,j_1}X_{i_1,j_2}\\
[X_{i_1,j_1},X_{i_2,j_2}]&=\de_{i_2,j_1}X_{i_1,j_2}-\de_{i_1,j_2}X_{i_2,j_1}\,.
\end{aligned}
\end{equation}
\eit

The Carnot group structure on the upper triangular matrices $N$ is defined by the grading $\fn=\oplus_iV_i$ where 
$$
V_i=\{A\in \gll(n,\R)\mid A_{jk}=0\;\text{if}\;k-j\neq i\}\,.
$$
For $r\neq 0$ the Carnot dilation $\de_r:N\ra N$ is given by conjugation with diagonal matrix $\diag(r^{-1},\ldots,r^{-n})$. 

\bigskip
\subsection*{The flag manifold}  The {\bf flag manifold $\f$} is the set of (complete) flags in $\R^n$, i.e. the collection of nested families of linear subspaces of $\R^n$
$$
W_1\subset\ldots \subset W_{n-1} 
$$
where $W_j$ has dimension $j$; this inherits a smooth structure as a smooth submanifold of the product of Grassmannians $G(1,n)\times\ldots\times G(n-1,n)$.

\begin{lemma}[Properties of $\f$]
\label{lem_properties_of_f}

\mbox{}
\ben
\item\label{item_transitive_actions} The action of $\gl(n,\R)\acts G(j,n)$ induces transitive actions 
$$
\gl(n,\R)\acts \f\quad\text{and}\quad G\acts \f\,.
$$
\item Letting $W_j^+:=\Span(e_1,\ldots,e_j)$, $W_j^-:=\Span(e_n,\ldots,e_{n-j+1})$, the stabilizer of the flag $(W_j^\pm)_{1\leq j\leq n-1}\in \f$ in $\gl(n,\R)$ is $P^\pm$.
\item The stabilizer of $(W_j^-)$ in   $G$ is $P$.    
 Henceforth we identify $\f$ with the homogeneous space (coset space) $G/P$ using the orbit map $g\mapsto g\cdot (W_j^-)$. 
\item The differential of the fibration $G\ra G/P\simeq \f$ at $e\in G$ yields an isomorphism 
$\fg/\fp\simeq T_P(G/P)\simeq T_{(W_j^-)}\f$.  
\item The fibration map $G\ra G/P$ is $P$-equivariant w.r.t. the action of  $P$ on  $G$ by conjugacy and on $G/P$ by left translation.
\item The differential at $e\in G$ induces an isomorphism of $P\stackrel{\Ad\restr_P}{\acts}\fg/\fp$ to the (isotropy) representation $P\acts T_P(G/P)$; here $G\stackrel{\Ad}{\acts} \fg$ is the Adjoint representation and $\Ad\restr_P$ is the restriction to $P$.
\item \label{item_fn_fg_fp_isomorphism} The composition $\fn\hookrightarrow \fg\ra \fg/\fp$ is an isomorphism.
\item \label{item_p_invariance_v_1} The image $\hat V_1$ of $V_1\subset\fn$ under $\fn\ra \fg/\fp\simeq T_P(G/P)$ is a $P$-invariant subspace of $T_P(G/P)$.  This  defines a $G$-invariant subbundle $\h\subset T(G/P)\simeq T\f$. 
\item \label{item_orbit_map_n} The orbit map $\al:N\ra \f$ given by  $\al(g)= g\cdot (W_j^-)$ 
is an $N$-equivariant embedding of $N$ onto an open subset of $\f$, which we denote by $\hat N\subset\f$.  The image $\hat N$ may be characterized as the collection of flags $(W_j)$ such that $W_j\cap W_{n-j}^+=\{0\}$ for every $1\leq j\leq n-1$. 
\item The embedding $\al$ is contact with respect to the subbundles $V_1\subset TN$ and $\h\subset T(G/P)$.
\item \label{item_delta_hat_delta} For every $r\in (0,\infty)$ we have $\al\circ\de_r=\hat \de_r\circ\al$, where $\hat\de_r:\f\ra\f$ is given by $\hat\de_r(x)=g_r\cdot x$ with $g_r=\diag(r^{-1},\ldots,r^{-n})$.
\een
\end{lemma}
\begin{proof}
(\ref{item_transitive_actions})-(\ref{item_fn_fg_fp_isomorphism}) follow readily from linear algebra or basic theory of manifolds. 

(\ref{item_p_invariance_v_1}).  To see this, it suffices to show that the image of $V_1\subset \fn$ under $\fn\hookrightarrow \gll(n,\R)/\fp^-$ is invariant under $\Ad\restr_{P^-}$.  Since $P^-$ is connected, it therefore suffices to see that $[\fp^-,V_1]\subset V_1+\fp^-$;    this follows readily by applying \eqref{eqn_bracket_relations}. 

(\ref{item_orbit_map_n}).  The orbit map $\al:N\ra \f$ is smooth and $N$-equivariant.  Since the differential of $\al$ at $e\in N$ is an isomorphism $D\al(e):\fn\ra \fg/\fp$,  the map $\al$  is an $N$-equivariant immersion onto an open subset.  But $N\cap P^-=\{e\}$ so it is also injective, and hence $\al$ is an embedding.  

Let $Z:=\{(W_j)\in \f\mid W_j\cap W_{n-j}^+=\{0\}\;\text{for all}\;1\leq j\leq n-1\}$.    If $(W_j)$ belongs to the orbit $\hat N=N\cdot (W_j^-)$ then $(W_j)\in Z$ since $W_j^-\cap W_{n-j}^+=\{0\}$ and $N$ fixes $W_{n-j}^+$ for all $j$. Now suppose $(W_j)\in Z$.  We prove by induction on $k$ that the $N$-orbit of $(W_j)$ contains a flag $(W_j^k)$ such that $W_j^k=W_j^-$ for all $j\leq k$.  Since $W_1\cap W_{n-1}^+=\{0\}$ we may choose $v=(v_1,\ldots,v_n)\in W_1$ with $v_n=1$.  If $n\in N$ is the block matrix
\begin{equation*}
\left[
\begin{matrix}
I& -b\\0&1
\end{matrix}
\right]
\end{equation*}
with $b=[v_1,\ldots,v_{n-1}]^t$ we have $n\cdot v\in W_1^-$, so letting $(W_j^1):=n\cdot (W_j)$ we have $W_1^1=W_1^-$, establishing the $k=1$ case.   The inductive step follows similarly, by working in the quotient $\R^n/W_{k-1}^-$.

(11).  Both $V_1$ and $\h$ are $N$-invariant subbundles, and $\al$ is $N$-equivariant, so the contact property follows from the fact that $\h(P)\subset T_P(G/P)$ is the image of $V_1\subset \fn$ under $D\al(e)$.

(\ref{item_delta_hat_delta}). Since $g_r\in P^-$ we have
$$
\al(\de_r(g))=\de_r(g)\cdot P=g_rgg_r^{-1}\cdot P=g_rg\cdot P=\hat\de_r(\al(g))\,.
$$

\end{proof}

\bigskip
\subsection*{Automorphisms} We collect some properties of automorphisms of $G$ and graded automorphisms of $N$.

\begin{lemma}[Properties of $\aut(G)$]
\label{lem_properties_aut_g}

\mbox{}
\ben
\item Transpose inverse $T\mapsto (T^t)^{-1}$ descends to a Lie group automorphism of $G$.   The group $\aut(G)$ of Lie group automorphisms of $G$ is generated by inner automorphisms and transpose inverse.
\item Let $\tau:\gl(n,\R)\ra \gl(n,\R)$ be the automorphism given by 
$$
\tau(T)=\Pi (T^t)^{-1}\Pi^{-1}
$$ 
where $\Pi\in\gl(n,\R)$ is the permutation matrix with $\Pi e_i=e_{n-i+1}$.  Then the induced automorphism  $\gll(n,\R)\ra\gll(n,\R)$ -- which we also denote by $\tau$ by abuse of notation -- is given by 
\begin{equation}
\label{eqn_tau_on_fg}
(\tau(A))_{i,j}=-A_{n-j+1,n-i+1}\,.
\end{equation} 
Then $\tau$ preserves $N$, and induces a graded automorphism $\fn\ra\fn$.
\item The automorphism $\tau:\gl(n,\R)\ra\gl(n,\R)$ descends to an automorphism $G\ra G$, which we also denote by $\tau$ (by further abuse of notation).    
\een
\end{lemma}
\begin{proof}~

(1). This follows from Theorem~\ref{diedo} since the only automorphism of $\R$ is the identity map.  
   Therefore there exists $\Phi_1\in \aut(G)$ in the group generated by inner automorphisms and transpose-inverse, such that $D\Phi_1=D\Phi$.  Hence $\Phi_2:=\Phi_1^{-1}\circ\Phi\in \aut(G)$ is the identity on the connected component of $e\in G$.  By considering centralizers of elements $g\in G$ of order $2$, it is not hard to see that $\Phi_2=\id$, and hence $\Phi=\Phi_1$.

(2) and (3) are straightforward.
\end{proof}

\bigskip\bigskip 
We now consider the group $\aut_{gr}(N)$ of graded automorphisms of $N$.

\bigskip

\begin{lemma}[Properties of $\aut_{gr}(N)$]
\label{lem_properites_aut_gr_n}
Assume $n\geq 4$.
\ben
\item \label{item_phi_inducing_graded_automorphism} If $\Phi\in \aut(G)$ and $\Phi(N)=N$, then $\Phi$ induces a graded automorphism of $N$ if and only if  $\Phi\restr_N=\tau^\eps\circ I_g\restr_N$    for some $\eps\in \{0,1\}$, $g=\diag(\la_1,\ldots,\la_n)\in G$.
\item Every graded automorphism of $N$ arises as in (1).
\item  For $1\leq j\leq n-1$ let $\fk_j\subset \fn$ be the Lie subalgebra generated by $\{X_{i,i+1}\}_{i\neq n-j}$, and $K_j\subset N$ be the Lie subgroup with Lie algebra $\fk_j$.
  A graded automorphism $N\ra N$ is induced by conjugation by some $g=\diag(\la_1,\ldots,\la_n)$ if and only if it preserves the subgroups $K_j$ for $1\leq j\leq n-1$.
\een
\end{lemma}
\begin{proof} (\ref{item_phi_inducing_graded_automorphism}).  Suppose $\Phi=\tau^\eps\circ I_g$ for some $g=\diag(\la_1,\ldots,\la_n)$.  Note that $I_g$ induces a graded automorphism of $N$ since
$$
I_g\circ \de_r=I_g\circ I_{g_r}=I_{gg_r}=I_{g_rg}=I_{g_r}\circ I_g=\de_r\circ I_g\,.
$$
Combining with Lemma~\ref{lem_properties_aut_g}(2) we get that $\Phi$ induces a graded automorphism of $N$, proving the ``if'' direction.

Now suppose $\Phi\in \aut(G)$ preserves $N$, and induces a graded automorphism $N\ra N$.

By  Lemma~\ref{lem_properties_aut_g}(1)
we have $\Phi=\tau^{\eps}\circ I_g$ for some $\eps\in\{0,1\}$, $g\in G$.  By Lemma~\ref{lem_properties_aut_g}(2), after postcomposing with $\tau$ if necessary, we may assume without loss of generality that $\Phi=I_g$ is an inner automorphism.  The condition $\Phi(N)=N$ is equivalent to saying that $g$ belongs to the normalizer of $N$ in $G$.     This is just the image of $P^+$ under the quotient map $\gl(n,\R)\ra G$.      (To see this, note that if $\hat g\in \gl(n,\R)$ is a lift of $g$, then $\hat g$ normalizes $N\subset\gl(n,\R)$.  But $\Span(e_1)$ is the unique fixed point of the action $N\acts G(1,n)$, so $\hat g$ fixes $\Span(e_1)$; passing to the quotient $\R^n/\Span(e_1)$ and arguing by induction, one gets that $\hat g\in P^+$. )
  Furthermore, after multiplying $g$ by a diagonal matrix, we may assume without loss of generality that $g\in N$.  

Now $I_g:N\ra N$ is a graded automorphism, so for all $r>0$ we have $I_g\circ \de_r=\de_r\circ I_g$ , i.e. 
\begin{align*}
&I_g\circ I_{g_r}=I_{g_r}\circ I_g\\
\implies &I_{g_rgg_r^{-1}g^{-1}}=\id_N\\
\implies &g_rgg_r^{-1}g^{-1}\in \cent(N)\\
\implies &g_rXg_r^{-1}-X\in \cent(\fn)
\end{align*} 
where $g=\exp X$ for $X\in \fn$; here we are using the fact that the exponential $\exp:\fn\ra N$ is a diffeomorphism.   Calculating in $\gll(n,\R)$ and letting $X=\sum_{j>i}a_{ij}X_{ij}$ we have
$$
g_rXg_r^{-1}-X=\sum_{j>i}a_{ij}(r^{j-i}-1)X_{ij}\in \cent(\fn)\,.
$$
Therefore $X\in \cent(\fn)$, and $g=\exp X\in \cent(N)$, and $I_g=\id$.  

(2).   Let $\phi:\fn\ra \fn$ be a graded automorphism of $\fn$.  We will use the fact that for any $x\in V_1$ and any $j\ge 1$,  we have $\phi\circ (\ad x\restr_{V_j})=(\ad\, \phi x\restr_{V_j})\circ \phi$, and in particular 
 $\rank(\ad\, x|_{V_j})=
 \rank(\ad\, \phi x|_{V_j})$.

First suppose $n=4$.

Set 
\begin{equation}
\begin{aligned}
\label{eqn_x_y_z_defs}
X_0=X _{23}\,,\;X_1=X _{12}\,,\;X_2=X _{34}\,,\\
Y_1=X _{13}\,,\;Y_2=-X _{24}\,,\;Z=X _{14}
\end{aligned}
\end{equation}   
Then  the following are the only nontrivial bracket relations between the basis elements:    
\begin{equation}
\label{eqn_x_y_z_brackets}
\begin{aligned}~
[X_0,X_1]=-Y_1, \;\; [X_0, X_2]=-Y_2,\\
[X_1, Y_2]=-Z=[X_2, Y_1].
\end{aligned}
\end{equation}
 Let $W_0:=\Span(X_0)$, and $W:=\Span(X_1, X_2)$.  If $x\in V_1$ then $\rank(\ad x\restr_{V_1})\leq 1$ iff $x\in W$, and $\rank(\ad x\restr_{V_2})=0$ iff $x\in W_0$.  It follows that $\phi (W)=W$ and $\phi (W_0)=W_0$, and so   
$$\phi X_0=a_0X_0\,,\quad \phi X_1=a_{11}X_1+a_{21}X_2\,,\quad \phi X_2=a_{12}X_1+a_{22}X_2$$
for some $0\not=a_0\in \R$,  $(a_{ij})\in \gl(2,\R)$.  Hence
\begin{align}
\notag \phi Y_1&=[\phi X_1, \phi X_0]=a_0a_{11}Y_1+a_0a_{21}Y_2\\
\notag \phi Y_2&=[\phi X_2, \phi X_0]=a_0a_{12}Y_1+a_0a_{22}Y_2\\
\label{eqn_ax1_ay1} 0&=[\phi X_1,\phi Y_1]=-2a_0a_{11}a_{21}Z\\
\label{eqn_ax2_ay2}0&=[\phi X_2,\phi Y_2]=-2a_0a_{12}a_{22}Z.
\end{align}

{\em Case 1.  $a_{11}\neq 0$.}   By (\ref{eqn_ax1_ay1}) we have $a_{21}=0$, and since $\det(a_{ij})\neq 0$ it follows that $a_{22}\neq 0$, and so $a_{12}=0$ by (\ref{eqn_ax2_ay2}).  Thus $\phi \restr_{V_1}$ is diagonal in the basis $X_0,X_1,X_2$.  Hence it is induced by an inner automorphism $I_g:G\ra G$, where $g$ is diagonal.

{\em Case 2. $a_{11}=0$.}  Since $\det(a_{ij})\neq 0$ it follows that $a_{12}$ and $a_{21}$ are both nonzero.   Then (\ref{eqn_ax2_ay2}) gives $a_{22}=0$.  Therefore modulo post-composition with the automorphism $\tau$ we get that $\phi \restr_{V_1}$ is diagonal, and so (2) holds.

Now suppose $n\geq 5$. 

First observe that
$$\{x\in V_1\mid   \text{rank}(\text{ad}\; x|_{V_1})=1\}=(\R X_{1,2}\cup \R X_{n-1,n})\backslash\{0\}.$$
  So $\phi  $ either  preserves or switches the two lines $\R X_{1,2}$, $\R X_{n-1,n}$, 
   and therefore after postcomposing with $\tau$ if necessary, we may assume without loss of generality that $\phi$ preserves them.
   Then we observe that 
   $$\{x\in V_1\mid   \text{rank}(\text{ad}\; x|_{V_{n-2}})=0\}=\text{span}\{X_{2,3}, \cdots, X_{n-2, n-1}\}.$$
    So $\phi  (\mathfrak h_0)=\mathfrak h_0$,  where $\mathfrak h_0$   is  the Lie subalgebra
       of $\fn$ generated by 
     $X_{2,3}, \cdots, X_{n-2, n-1}$. Note $\mathfrak h_0$ is  isomorphic to $\fn_{n-2}$, i.e. the strictly upper triangular matrices in $\gll(n-2,\R)$.  
  
Suppose $n=5$.  In this case $\mathfrak h_0{\cap V_1}=\text{span}\{X_{2,3}, X_{3,4}\}$.  We have 
 $$\{x\in \mathfrak h_0\cap V_1\mid [x, X_{1,2}]=0\}=\R X_{3,4}$$
  and 
  $$\{x\in \mathfrak h_0\cap V_1\mid [x, X_{4,5}]=0\}=\R X_{2,3}.$$

Because $\phi  $ preserves the two lines $\R X_{1,2}$, $\R X_{4,5}$ then it also preserves the two lines $\R X_{2,3}$, $\R X_{3,4}$.  Hence $\phi$ is diagonal in the basis $\{X_{i,i+1}\}_{1\leq i\leq 4}$, and  therefore (2) holds.

Now we assume $n\ge 6$ and that (2) holds for $\fn_k$ with $k\le n-2$.  
 We already observed  $\phi  (\mathfrak h_0)=\mathfrak h_0$, and have reduced to the case that the two lines $\R X_{1,2}$, $\R X_{n-1,n}$ are preserved by $\phi$.  Since $\mathfrak h_0$ is isomorphic to $\fn_{n-2}$,  the  induction hypothesis  implies that either 
  $\phi  (\R X_{i, i+1})=\R X_{i, i+1}$ for all $i=2,\cdots, n-2$ or 
$\phi  (\R X_{i, i+1})=\R X_{n-i, n+1-i}$  for all $i=2,\cdots, n-2$.  
Since $[X_{1,2}, X_{2,3}]\not=0$ and 
$[X_{1,2}, X_{n-2,n-1}]=0$, and $\phi  $ is an automorphism of Lie algebra,   we see that 
 $\phi  (\R X_{i, i+1})=\R X_{i, i+1}$ for all $i=1,\cdots, n-1$,
and so (2) holds.

(3). Let $\Phi$ be a graded automorphism, so by (2) we have $\Phi=\tau^{\eps}\circ I_g$ for $\eps\in \{0,1\}$ and some $g$ diagonal.   If $\eps=0$ then $D\Phi(X_{i,i+1})\in \R X_{i,i+1}$, so $\Phi$ preserves $K_j$ for all $j$.  On the other hand, by \eqref{eqn_tau_on_fg} we have $\tau(K_j)=K_{n-j}$, so if $\Phi(K_j)=K_j$ for all $1\leq j\leq n-1$ then $\eps=0$.

\end{proof}

\bigskip
We now consider the action of $\aut(G)$ on the flag manifold.

\begin{lemma}
\label{lem_action_aut_g_on_f}

\mbox{}
\ben
\item There is a 1-1 correspondence between cosets of $P$ and their stabilizers in $G$ with respect to the action $G\acts G/P$: if $g_1,g_2\in G$ then 
\begin{align*}
\stab(g_1P)=\stab&(g_2P)\quad \iff \quad g_1Pg_1^{-1}=g_2Pg_2^{-1}\\
&\iff g_1P=g_2P\,.
\end{align*}
\item Every $\Phi\in\aut(G)$ permutes the conjugates of $P$; by (1) we thereby obtain an action of $\aut(G)\stackrel{\rho}{\acts}G/P$ 
defined by 
\begin{equation}
\label{eqn_aut_g_action_def}
\stab(\Phi\cdot gP)=\Phi(\stab(gP))\,.
\end{equation}  
More explicitly, if $\Phi(P)=hPh^{-1}$ then  
\begin{equation}
\label{eqn_aut_g_action_def_2}
\Phi\cdot gP=\Phi(g)hP\,;
\end{equation} 
in particular $\rho(\Phi)$ defines a smooth diffeomorphism.
\item For every $\Phi\in \aut(G)$ the map 
$$
\rho(\Phi):G/P\lra G/P
$$
is $G$-equivariant from $G\stackrel{\ell}{\acts}G/P$ to $G\stackrel{\ell_\Phi}{\acts}G/P$, where $\ell(\bar g)( gP):=\bar ggP$ and $\ell_\Phi(\bar g)(gP):=\Phi(\bar g)gP$.
\item The action $\aut(G)\stackrel{\rho}{\acts} G/P$ preserves the horizontal subbundle $\h\subset T(G/P)$.
\item If $\Phi_0\in \aut(G)$ is transpose-inverse, then $\Phi_0\cdot (W_j)=(W_j^\perp)$ for every $(W_j)\in \f$, i.e. $\rho(\Phi_0)=\psi$.
\een
\end{lemma}
\begin{proof}(1). Note that the normalizer of $P$ in $G$ is $P$ itself.  Therefore 
\begin{align*}
\stab&(g_1P)=\stab(g_2P)\quad \iff \quad g_1Pg_1^{-1}=g_2Pg_2^{-1}\\
&\iff (g_2^{-1}g_1)P(g_2^{-1}g_1)^{-1}=P\quad \iff\quad g_2^{-1}g_1\in P\\
&\iff g_1P=g_2P\,.
\end{align*}

(2).  Note that $\tau(P)=P$ so $\tau$ permutes the conjugates of $P$; by Lemma~\ref{lem_properties_aut_g}(1) it follows that every $\Phi\in\aut(G)$ permutes the conjugates of $P$.  Since $\stab(gP)=gPg^{-1}$, if $\Phi(P)=hPh^{-1}$, then for every $g\in G$ 
\begin{align*}
\stab(\Phi\cdot gp)=\Phi(\stab(gP))=\Phi(gPg^{-1})\\
=\Phi(g)hPh^{-1}(\Phi(g))^{-1}=\stab(\Phi(g)hP)
\end{align*}
so $\Phi\cdot gP=\Phi(g)hP$.

(3).  For $g,\bar g\in G$, by (2) we have
\begin{align*}
\Phi\cdot(\ell(\bar g)\cdot gP)=\Phi\cdot(\bar ggP)=\Phi(\bar gg)hP\\
=\ell(\Phi(\bar g))\cdot(\Phi(g)hP)=\ell_\Phi(\bar g)\cdot(\Phi\cdot gP)
\end{align*}
as claimed.

(4).  If $\Phi=I_{\bar g}$ is an inner automorphism, then $I_{\bar g}\cdot gP=\bar ggP$ by \eqref{eqn_aut_g_action_def_2}; hence $I_{\bar g}\cdot$ preserves $\h$ by Lemma~\ref{lem_properties_of_f}(9).  In the case $\Phi=\tau$, we have $\tau(P)=P$, so by \eqref{eqn_aut_g_action_def_2} we have $\tau\cdot gP=\tau(g)P$, and since $\tau(X_{i,i+1})=X_{n-i,n-i+1}$ we have $\tau(V_1)=V_1$.  It follows that the differential of $\tau\cdot : G/P\ra G/P$ at $P$ preserves $\h(P)$.  Using (3) and the fact that the actions $\ell$ and $\ell_\tau$ preserve $\h$, we conclude that $\tau\cdot$ preserves $\h$.    By Lemma~\ref{lem_properties_aut_g} the group $\aut(G)$ is generated by $\tau$ and the inner automorphisms, so (4) follows.

(5).  Let $\Pi\in \gl(n,\R)$ be the permutation matrix with $\Pi e_i=e_{n-i}$, so $\Phi_0(P)=\Pi P\Pi^{-1}$.  If $(W_j)=g\cdot (W_j^-)$, then using \eqref{eqn_aut_g_action_def_2} we have
\begin{align*}
\Phi_0\cdot (W_j)&=\Phi_0\cdot (g\cdot(W_j^-))=(\Phi_0(g)\Pi)\cdot(W_j^-)\\
&=((g^{-1})^t\Pi W_j^-)=((g^{-1})^tW_j^+)\,.
\end{align*}
Since $(g^{-1})^tW_j^+$ is orthogonal to $gW_{n-j}^-$ for every $1\leq j\leq n-1$, assertion (5) follows.

\end{proof}

\bigskip
\subsection*{Fibrations between flag manifolds}
We now consider partial flags and the associated flag manifolds.

If $\Si\subset \{1,\ldots,n-1\}$, we let $\f_{\Si}$ be the collection of partial flags $(W_j)_{j\in \Si}$ where $W_j\subset \R^n$ has dimension $j$.  Then, as in the case when $\Si=\{1,\ldots,n-1\}$, the set $\f_{\Si}$ is a smooth submanifold of the product $\prod_{j\in \Si}G(j,n)$ of Grassmannians, and the actions $G\acts G(j,n)$ yield a transitive smooth action $G\acts \f_{\Si}$.  Taking $(W_j^-)_{j\in \Si}$ to be the basepoint, and letting $P_{\Si}=\stab((W_j^-)_{j\in \Si})\subset G$ be its stabilizer, we obtain a $G$-equivariant diffeomorphism $G/P_{\Si}\ra \f_{\Si}$; we identify $\f_{\Si}$ with $G/P_{\Si}$ accordingly.    We will be interested in the case when $\Si=\{j\}$ for some $j\in \{1,\ldots,n-1\}$, so to simpify notation we let $\f_j:=\f_{\{j\}}=G(j,n)$.

\begin{lemma}
\label{lem_properties_fibration_between_flag_manifolds}

\mbox{}
\ben
\item If $\Si_1\subset \Si_2\subset \{1,\ldots,n-1\}$ then we obtain a smooth $G$-equivariant fibration $\pi_{\Si_1,\Si_2}:\f_{\Si_2}\ra \f_{\Si_1}$ by ``forgetting subspaces'', i.e. $\pi_{\Si_1,\Si_2}((W_j)_{j\in \Si_2}):=(W_j)_{j\in \Si_1}$.  
\item The fiber of $\pi_{\Si_1,\Si_2}$ passing through $(W_j^-)$ is the $P_{\Si_1}$ orbit $P_{\Si_1}\cdot (W_j-)$, which is diffeomorphic to $P_{\Si_1}/P_{\Si_2}$.  
 To simplify notation, for $j\in \{1,\ldots,n-1\}$ we let  $\f_j:=\f_{\{j\}}=G(j,n)$ and
 $$\pi_j:=\pi_{\{j\},\{1,\ldots,n-1\}}:\f=\f_{\{1,\ldots,n-1\}}\ra \f_j\,,
$$
so the fiber of $\pi_j$ passing through the basepoint $P$ is $P_jP$. 
\item The intersection of the fiber $\pi_j^{-1}(W_j^-)$ with $\hat N=N\cdot (W_j^-)$ is $K_j\cdot (W_j^-)$, where $K_j\subset N$ is as in Lemma~\ref{lem_properites_aut_gr_n}(3).  Hence the orbit map $\al:N\ra \hat N$ carries the coset foliation of $K_j$ to the foliation defined by $\pi_j$.
\een
\end{lemma}
\begin{proof}

(1) and (2) are a special case of a general fact:  if $K$ is a Lie group and $H_1\subset H_2\subset K$ are closed subgroups, the quotient map $K/H_1\ra K/H_2$ is a $K$-equivariant fibration; fiber passing through $H_1$ is $H_2H_1\subset K/H_1$ and it is diffeomorphic to $H_2/H_1$.

(4).  The subbundle of $T(G/P)$ tangent to the fibers of the fibration $\pi_j:G/P\ra G/P_j$ is $G$-invariant, and its value at the basepoint $P\in G/P$ is the subspace $\fp_j/\fp$.  The subbundle of $TN$ tangent to the coset foliation of $K_j$ maps under the orbit map $\al:N\ra \hat N$ to an $N$-invariant subbundle of $T\hat N$ whose value at $P$ is the subspace $(\fk_j+\fp)/\fp$.  But $\fk_j+\fp=\fp_j$, so the two subbundles are the same.  It follows that $\pi_j^{-1}(W_j^-)\cap\hat N$ is the union of a closed set of $K_j$-orbits in $\hat N$; however, $\pi_j^{-1}(W_j)\cap\hat N$ is invariant under the action of $\hat \de_r$, so it can contain only one $K_j$-orbit.  
\end{proof}

\bigskip
\subsection*{Dynamics of Carnot dilations}~
The following result is a special case of \cite[Lemma 3.9]{tits_alternative}; we give a self-contained proof for the reader's convenience.
\begin{lemma}
\label{lem_dilation_dynamics}
Pick $1\leq j\leq n-1$.
\ben
\item Let $K_1\subset \f_1$ be a compact subset disjoint from $Z_1:=\{W_1\in \f_1\mid W_1\subset W_{n-j}^+\}$, 
and $U_1\subset\f_1$ be an open subset containing $Z_2:=\{W_1\in \f_1\mid W_1\subset W_j^-\}$.  Then there exists $\bar r_1=\bar r_1(K_1,U_1)>0$ such that for every $r\leq \bar r_1$ we have $\hat \de_r(K_1)\subset U_1$.
\item Let $K\subset \f_j$ be a compact subset such that  $W_j\cap W_{n-j}^+=\{0\}$ for every $W_j\in K$, and $U\subset\f_j$ be an open subset containing $W_j^-$.  Then there exists $\bar r=\bar r(K,U)>0$ such that for every $r\leq \bar r$ we have $\hat \de_r(K)\subset U$.
\een
\end{lemma}
\begin{proof}
(1).  Let $\pi_{W_{n-j}^+}$  and $\pi_{W_j^-}$ be the projections onto the summands of the decomposition $\R^n=W_{n-j}^+\oplus W_j^-$.
Define $\hat\be:\R^n\setminus W_{n-j}^+\ra [0,\infty)$ by $\hat\be(v):=\frac{\|\pi_{W_{n-j}^+}v\|}{\|\pi_{W_j^-}v\|}$.   Note that by homogeneity $\hat\be$ descends to $\be:\f_1\setminus Z_1\ra [0,\infty)$.

Since $K_1$ is compact and $U_1$ is open, we may choose $C_2\leq C_1<\infty$ such that $\be(K_1)\subset [0,C_1]$ and $\be^{-1}([0,C_2])\subset U$.  
If $v\in \R^n\setminus W_{n-j}^+$ and $r\leq 1$ then 
\begin{align*}
\|\pi_{W_{n-j}^+}(\hat\de_rv)\|&=\|\sum_{i=1}^{n-j}r^{-i}v_i\|\leq r^{-(n-j)}\|\pi_{W_{n-j}^+}v\|\\
\|\pi_{W_j^-}(\hat\de_rv)\|&=\|\sum_{i=n-j+1}^nr^{-i}v_i\|\geq r^{-(n-j+1)}\|\pi_{W_j^-}v\|\\
\implies&\hat\be(\hat\de_rv)\leq r\hat\be(v)\,.
\end{align*}
Hence for $W_1\in\f_1\setminus Z_1$ and $r\leq 1$ we have $\be(\hat\de_rW_1)\leq r\be(W_1)$.
Taking $\bar r_1:=C_2/C_1$, if $r\leq \bar r_1$ then 
$$
\hat\de_r(K_1)\subset\hat \de_r(\be^{-1}([0,C_1]))\subset \be^{-1}([0,C_2])\subset U_1\,.
$$

(2).  Note that $$K_1:=\{W_1\in\f_1\mid \exists W_j\in K\;\text{s.t.}\;W_1\subset W_j\}$$ is a compact subset of $\f_1\setminus Z_1$,   and there is an open subset $U_1\subset\f_1$  containing $Z_2$ such that 
if $W_j\in \f_j$ and $\{W_1\in\f_1\mid W_1\subset W_j\}\subset U_1$, then $W_j\in U$.
Let $\bar r=\bar r_1(K_1,U_1)$ be as in (1), and choose $W_j\in K$,  $r\leq \bar r$.  Then by the choice of $r$, for every $W_1'\in\f_1$ with $W_1'\subset W_j$ we have $\hat\de_r(W_1')\in U_1$.  Therefore $\{W_1\in \f_1\mid W_1\subset \hat\de_r(W_j)\}\subset U_1$, and hence $\hat \de_r(W_j)\in U$, by our assumption on $U_1$.
\end{proof}

\bigskip\bigskip
\subsection{Sobolev mappings and the Pullback Theorem}
\label{subsec_sobolev_mappings_pullback_theorem}~
We give a very brief discussion here, and refer the reader to \cite[Section 2]{KMX1} and the references therein for more details.

Let $H$, $H'$ be Carnot groups with gradings $\fh=\oplus_jV_j$, $\fh'=\oplus_j V_j'$, Carnot dilations denoted by $\de_r$, and Carnot-Caratheodory distance functions $d_{CC}$, $d_{CC}'$.  Let $\nu=\sum_jj\dim V_j$ be the homogeneous dimension of $H$.  

Choose $p>\nu$.   If $U\subset H$ is open and $f:U\ra H'$ is a continuous map then $f$ is a $W^{1,p}_{\loc}$-mapping if there exists $g\in L^p_{\loc}(U)$ such that for every $1$-Lipschitz function $\psi:H'\ra \R$ and every unit length $X\in V_1$, the distribution derivative $X(\psi\circ f)$ satisfies $|X(\psi\circ f)|\leq g$ almost everywhere.  If $f:H\supset U\ra U'\subset H'$ is a quasiconformal homeomorphism then $f$ is a $W^{1,p}_{\loc}$-mapping for some $p>\nu$.

If $f:U\ra H'$ is a $W^{1,p}_{\loc}$-mapping then for a.e. $x\in U$ the map $f$ has a Pansu differential, i.e. letting $f_x:=\ell_{f(x)^{-1}}\circ f\circ \ell_x$, $f_{x,r}:= \de_r^{-1}\circ f_x\circ \de_r$ there is a graded group homomorphism  $D_Pf(x):H\ra H'$ such that  
$$
f_{x,r}\stackrel{C^0_{\loc}}{\lra} D_Pf(x)
$$
as $r\ra 0$.   A Lie group homomorphism $\Phi: H\ra H'$ is graded if $\de_r\circ \Phi=\Phi\circ \de_r$ for all $r>0$.  We will often use $D_Pf(x)$ to denote the induced graded homomorphism of Lie algebras $\fh\ra \fh'$.  

We let $\{X_j\}_{1\leq j\leq \dim H}$ be a graded basis for $\fh$, and $\{\th_j\}_{1\leq j\leq \dim H}$ be the dual basis. 

The {\em weight} of a subset $I\subset \{1,\ldots,\dim H\}$ is given by 
$$
\wt I:=-\sum_{i\in I}\deg i
$$
where $\deg:\{1,\ldots,\dim H\}\ra \{1,\ldots, s\}$ is defined by $\deg i=j$ iff $X_i\in V_j$.  For a non-zero  left-invariant form
$\alpha = \sum_{I} a_I \theta_I$ we define $\wt(\alpha) = \max \{ \wt I : a_I \ne 0\}$ and set $\wt(0):=-\infty$; here $\th_I$ denotes the wedge product $\La_{i\in I}\th_i$.

We use primes for the objects associated with $H'$.

Fix $p>\nu$, and let $f:U\ra H'$ be a $W^{1,p}_{\loc}$-mapping for some open subset $U\subset H$.  If $\om\in\Om^k(H')$ is a differential $k$-form with continuous coefficients, the {\em Pansu pullback of $\om$} is the $k$-form with measurable coefficients $f_P^*\om$ given by 
$$
f_P^*\om(x):=(D_Pf(x))^*\om(f(x))\,,
$$ 
where  $D_Pf(x):\fh\ra \fh'$  is the Pansu differential of $f$ at $x\in U$.

 We will use the following special case of  \cite[Theorem 4.2]{KMX1}:
\begin{theorem}[Pullback Theorem (special case)]  \label{co:pull_back2}
Suppose $\varphi \in C_c^\infty(U)$ and  that $\alpha$ and $\beta$ are closed left invariant forms
which satisfy 
\begin{equation} \label{eq:weight_special}
\deg \alpha + \deg \beta = N -1 \quad \hbox{and} \quad  \wt(\alpha) + \wt(\beta) \le -\nu + 1.
\end{equation}
Then 
\begin{equation} \label{eq:pull_back_identity_special}
 \int_U f_P^*(\alpha) \wedge d(\varphi \beta) = 0.
\end{equation}
\end{theorem}

\bigskip
We now consider Sobolev mappings on the flag manifold.  

Let $f:\f\supset U\ra \f$ be a continuous map.  Then:
\bit
\item $f$ is a {\bf $W^{1,p}_{\loc}$-mapping} if for every $x\in U$ there is an open neighborhood $V$ of $x$ and group elements $g,g'\in G=\pgl(n,\R)$ such that $V\subset g\cdot \hat N$, $f(V)\subset g'\cdot\hat N$, and the composition 
\begin{equation}
\label{eqn_f_in_charts}
(\rho(g)\circ\al)^{-1}(V)\stackrel{\rho(g)\circ \al}{\lra}V\stackrel{f}{\lra}g'\cdot\hat N\stackrel{(\rho(g')\circ\al)^{-1}}{\lra}N
\end{equation}
is a $W^{1,p}_{\loc}$-mapping. 
\item The map $f$ is {\bf Pansu differentiable at $x\in U$} if for some $x\in U$, $V\subset U$, $g,g'\in G$ as above, the composition \eqref{eqn_f_in_charts} is Pansu differentiable at $(\rho(g)\circ\al)^{-1}(x)$.   By the chain rule for Pansu differentials, Pansu differentiability is independent of the choice of $g,g'$, and the resulting Pansu differential $\fn\ra \fn$ is well-defined up to pre/post composition with graded automorphisms.  In particular, the property being an isomorphism is well-defined.
\eit  
Equivalently,  one may work directly with the flag manifold as an equiregular subriemannian manifold, and use the notion of Pansu differential in that setting  (see \cite[Section 1.4]{gromov_carnot_caratheodory_space_seen_from_within}, \cite{vodopyanov_carnot_manifolds}, and \cite[Appendix A]{KMX1}).

\bigskip

\section{Sobolev mappings and the (virtual) preservation of coset foliations}~
\label{sec_sobolev_mappings_preservation_foliations}

The main result in this section is the following:

\begin{lemma}
\label{lem_main_no_oscillation}
Let $n\ge 4$. 
Let $U\subset N$ be a connected open subset, and for some $p>\nu$ let $f:N\supset U\ra N$ be a $W^{1,p}_{\loc}$-mapping whose Pansu differential is an isomorphism almost everywhere.  Then after possibly composing with $\tau$, if necessary, 
for a.e. $x\in U$ the Pansu differential $D_Pf(x)$ preserves the subspace $\R X_{i,i+1}\subset V_1$ for every $1\leq i\leq n-1$.
\end{lemma}

\begin{corollary}
\label{cor_preservation_coset_foliation}
Let  $n\ge 4$ and  $f:U\ra N$ be as in Lemma~\ref{lem_main_no_oscillation}, and $K_j$ be as in Lemma~\ref{lem_properites_aut_gr_n}(3).  Then after possibly composing with $\tau$, if necessary,  for every $1\leq j\leq n-1$, $f$ locally preserves the coset foliation of $K_j$, i.e. for every $x\in U$ there is an $r>0$ such that for every $g\in N$ the image of $gK_j\cap B(x,r)$ under $f$ is contained in a single coset of $K_j$.
\end{corollary}
The corollary follows from Lemma~\ref{lem_main_no_oscillation} and \cite[Lemma 2.30]{KMX2}.  Note that Lemma~\ref{lem_main_no_oscillation} fails when $n=3$ (when $N$ is a copy of the Heisenberg group), even for automorphisms.  Examples from \cite{kmsx_infinitesimally_split_globally_split} show that Lemma~\ref{lem_main_no_oscillation} can fail even for bilipschitz mappings whose Pansu differential preserves the splitting $V_1=\R X_{1,2}\oplus \R X_{2,3}$ for a.e. $x$.

\bigskip
\begin{proof}[Proof of Lemma~\ref{lem_main_no_oscillation}]  We begin with the $n=4$ case of the lemma. 

As in the proof of Lemma~\ref{eqn_x_y_z_brackets}, we let 
\begin{equation*}
\begin{aligned}
X_0=X _{23}\,,\;X_1=X _{12}\,,\;X_2=X _{34}\,,\\
Y_1=X _{13}\,,\;Y_2=-X _{24}\,,\;Z=X _{14}\,.
\end{aligned}
\end{equation*}   
 Let $\alpha_0, \alpha_1, \alpha_2$, $\beta_1, \beta_2$, $\gamma$ be the basis 
 of left invariant $1$-forms that are  dual to $X_0, X_1, X_2$, $Y_1, Y_2$, $Z$. Then using \eqref{eqn_x_y_z_brackets} we have
     \begin{align*}
      d\alpha_0=d\alpha_1=d\alpha_2=0;\\
     d \beta_1=\alpha_0\we\alpha_1,\;\;\; d \beta_2=\alpha_0\we \alpha_2;\\
     d\gamma=\alpha_1\we \beta_2+\alpha_2\we\beta_1.
     \end{align*}

Let  $\om:=\al_1\we\be_1\we\ga$.  By Lemma \ref{lem_properites_aut_gr_n}(2)  on graded automorphisms, 
$$
f_P^*\om=u_1\al_1\we \be_1\we\ga+u_2\al_2\we \be_2\we \ga
$$
where $u_1$, $u_2$ are measurable and with $S_i:=\{u_i\neq 0\}$ the union $S_1\cup S_2$ has full measure and $S_1\cap S_2$ is null.

Let $\eta=\alpha_2\wedge\beta_2$  and pick $\varphi\in C^\infty_c(U)$. Then $\omega$ and $\eta$ are closed left invariant forms  and their degrees and weights satisfy the assumption of  
 the pullback theorem (Theorem \ref{co:pull_back2}),   and so    
 $$\int_U f_P^* \omega\wedge d(\varphi \eta)=0\,.$$      As $ d(\varphi \eta)=d\varphi\wedge \eta$  we have 
     $f_P^* \omega\wedge d(\varphi \eta)=\pm u_1(X_0\varphi) d\text{vol}$  
   and therefore     $\int_U u_1(X_0\varphi)d \vol=0$.   Since $\varphi$ was arbitrary, it follows that   $X_0u_1=0$ in the sense of distributions.  Similarly by picking $\eta=\alpha_0\wedge\beta_2$    we obtain  $X_2u_1=0$.  It follows that $Y_2u_1=[X_2,X_0]u_1=X_2X_0u_1-X_0X_2u_1=0$.  
      Let $H_i$ ($i=1, 2$) be the Lie subgroup of $N$ whose Lie algebra is generated by 
         $X_0$ and $X_i$.  Then we  see that for almost every left coset $gH_2$,  $u_1$ is  locally constant almost everywhere.   
    Consequently  
   $X_0\chi_{S_1}=X_2\chi_{S_1}=Y_2\chi_{S_1}=0$, where $\chi_S$ denotes the characteristic function of a subset $S\subset U$.  
Similarly  by using $\eta=\alpha_1\wedge \beta_1$, $\alpha_0\wedge \beta_1$,   we obtain 
$X_0u_2=X_1u_2=Y_1u_2=0$  and $X_0\chi_{S_2}=X_1\chi_{S_2}=Y_1\chi_{S_2}=0$.    As  $\chi_{S_1}+\chi_{S_2}=1$, we infer  that $X_i\chi_{S_j}=0$ for all 
 $0\le i\le 2$ and $j=1, 2$.    Since $X_0, X_1, X_2$ generates $\mathfrak n$, this gives $X\chi_{S_j}=0$ for every $X\in \fn$, and hence $\chi_{S_j}$ is  locally constant a.e.  By the connectedness of $U$, $\chi_{S_j}$ is  constant  a.e.  
 As $\chi_{S_1}$ takes on only two possible values ($0$ or $1$), we have 
$\chi_{S_1}=0$  a.e. or $\chi_{S_1}=1$  a.e.
 If  $\chi_{S_1}=1$  a.e., then  by  the definition of  $S_1$ we have 
 $Df_x(\mathbb R X_{i})=\mathbb R X_{i}$ for all $0\le i\le 2$ and a.e. $x\in U$.  If 
 $\chi_{S_1}=0$  a.e., then 
 $\tau\circ Df_x(\mathbb R X_{i})=\mathbb R X_{i}$ for a.e. $x\in U$.



Now assume $n\geq 5$.  

Let $\{\th_{i,j}\mid 1\le i<j\le n\}$ be  the basis of left invariant $1$-forms on $N$ that are dual to the basis  $\{X_{i,j}\mid 1\le i<j\le n\}$ of $\fn$.  Then we have 
       \begin{equation}
       \label{eqn_dthij}
       d \th_{i,j}=\left\{
       \begin{array}{rl}
       0& \text{  if    }   j=i+1\\
       -\sum_{k=i+1}^{j-1}\th_{i, k}\we \th_{k,j}  & \text{  if   } j>i+1
       \end{array}\right.
       \end{equation}
   
Let $\omega_+=\th_{1,2}\we\th_{1,3}\we\cdots\we \th_{1,n}$ and
$\omega_-=\th_{n-1,n}\we\th_{n-2,n}\we\cdots\we \th_{1,n}$.  Note that 
$\omega_+$ is closed since  by \eqref{eqn_dthij} we have $d\th_{1,2}=0$,   and   
 $d\th_{1,j}=-\sum_{k=2}^{j-1}\th_{1,k}\we\th_{k,j}$    for $j\ge 3$ and  so 
  $\th_{1,2} \we\cdots \we \th_{1, j-1}\we d\th_{1,j}=0$.   Similarly 
 $\omega_-$ is  closed.  By  Lemma~\ref{lem_properites_aut_gr_n}  we have 
$$f^*_P(\omega_+)=u_+\omega_++u_-\omega_-$$
 for some measurable functions $u_+$, $u_-$.

 For $2\le k\le n-1$, let 
 $$\eta_{k-}=\bigwedge_{2\le i<j\le n,  (i,j)\not=(k, k+1)} \th_{i,j},$$
  and for  $1\le k\le n-2$, let 
 $$\eta_{k+}=\bigwedge_{1\le i<j\le n-1,  (i,j)\not=(k, k+1)} \th_{i,j}.$$
  We show  that  $\eta_{k-}$ is closed.  First we have  $d\th_{l, l+1}=0$.   For $m>l+1$, 
 $d\th_{l, m}=-\sum_{s=l+1}^{m-1}\th_{l, s}\we\th_{s, m}$
   and  at least one of the two terms $\th_{l, s}$, $\th_{s, m}$  is already present in 
 $$\bigwedge_{2\le i<j\le n,  (i,j)\not=(k, k+1), (l,m)} \th_{i,j};$$  it follows that
$d\th_{l, m}\we \bigwedge_{2\le i<j\le n,  (i,j)\not=(k, k+1), (l,m)} \th_{i,j}=0$.   
Similarly  $\eta_{k+}$ is closed.  
  We next show that  the  degree and weight conditions in the pullback theorem  (Theorem \ref{co:pull_back2})    are satisfied: for $\omega_+$   and $\eta=\eta_{k-}$ or $\eta_{k+}$  we have  $\text{degree}(\omega_+)+\text{degree}(\eta)=N-1$ and $\text{wt}(\omega_+)+\text{wt}(\eta)=-\nu+1$.  
    First these conditions are satisfied by $\omega_+$ and $\eta_-$ as 
        $\omega_+\wedge \eta_-=\pm\bigwedge_{(i,j)\not={k, k+1}}\theta_{ij}$ and 
          $\theta_{k, k+1}$ has degree one and weight $-1$.   These conditions are also satisfied by $\omega_+$ and $\eta_+$   as $\tau(\eta_-)=\pm\eta_+$ and so  $\eta_-$, $\eta_+$ have the same degree and weight. 
Hence  by Theorem \ref{co:pull_back2},   for any $\varphi\in C^\infty_c(U)$, we have $\int_U f_P^*(\omega_+)\wedge d(\varphi \eta)=0$.  
 As  $d(\varphi \eta)=d\varphi\wedge \eta$,  
  we get $f_P^*(\omega_+)\wedge d(\varphi \eta_{k-})=\pm u_+X_{k,k+1}\varphi \text{Vol}$    (for $2\le k\le n-1$)  and 
 $f_P^*(\omega_+)\wedge d(\varphi \eta_{k+})=\pm u_-X_{k,k+1}\varphi \text{Vol}$  (for $1\le k\le n-2$). 
 It follows   that distributionally we have  
       $X_{k,k+1} u_+=0$  for all $2\le k\le n-1$  and 
          $X_{k,k+1} u_-=0$  for all $1\le k\le n-2$.

Let $S_+\subset U$ be the subset where the Pansu differential  of $f$   fixes all the directions $\R X_{i, i+1}$.  
  Similarly let $S_-\subset U$ be the set of points $x\in U$ such that 
    $\tau\circ Df_x$  fixes all the directions $\R X_{i, i+1}$. Again  by Lemma \ref{lem_properites_aut_gr_n} (2),
     $S_+\cup S_-$ has full measure in $U$ and  $S_+\cap S_-$  is null.    
Notice that 
 $S_+$ agrees with $\{x\in U\mid u_+(x)\not=0\}$ up to a set of measure $0$  and  $S_-$ agrees with  $\{x\in U\mid u_-(x)\not=0\}$ up to a set of measure $0$.  Hence $X_{k,k+1} \chi_{S_+}=0$ for all $2\le k\le n-1$   and  $X_{k,k+1} \chi_{S_-}=0$ for all $1\le k\le n-2$.
  Since $\chi_{S_+}+\chi_{S_-}=1$ we have $X_{k,k+1} \chi_{S_+}=0$ for all $1\le k\le n-2$.
 So    $X_{k,k+1} \chi_{S_+}=0$ for all $1\le k\le n-1$.   Since the $X_{k, k+1}$, $1\le k\le n-1$ generates the Lie algebra  $\mathfrak n$, we see that  $X \chi_{S_+}=0$ for any $X\in \mathfrak n$ and   we conclude as in the $n=4$ case.

 

\end{proof}

\bigskip
\section{Rigidity for fibration preserving maps}
\label{sec_rigidity_foliation_preserving_maps}

In the section we show that mappings of $\f$ which respect the fibrations $\pi_j:\f\ra \f_j$ for all $1\leq j\leq n-1$ are locally projective.

\begin{definition}
A mapping $f:\f\supset U\ra \f$ is {\bf fibration-preserving} if for every $1\leq j\leq n-1$, and every $W_j\in \f_j$, the image of $U\cap \pi_j^{-1}(W_j)$ under $f$ is contained in $\pi_j^{-1}(W_j')$ for some $W_j'\in\f_j$.  Equivalently, letting $U_j:=\pi_j(U)$, there exist mappings $f_j:U_j\ra \f_j$ such that $\pi_j\circ f=f_j\circ \pi_j$.  A mapping $f:\f\supset U\ra \f$  is {\bf locally fibration preserving} if it is fibration preserving near any point in $U$.
\end{definition}

\begin{proposition}
\label{thm_fiber_preserving_projective}~
Suppose $n\geq 3$ and $f:U\ra \f$ is a continuous, locally fiber preserving mapping, where $U\subset \f$ is a connected open subset.  If   for some $x\in U$ the Pansu differential $D_Pf(x)$ exists and is an isomorphism (see Subsection~\ref{subsec_sobolev_mappings_pullback_theorem}) then $f$ agrees with $g:\f\ra \f$  for some $g\in \pgl(n,\R)$.    
\end{proposition}

\bigskip
\subsection{Proof of Proposition~\ref{thm_fiber_preserving_projective} in the $n=3$ case}
\label{subsec_fibration_preserving_n_3_case}
In this subsection we fix $n=3$, and without further mention $\f$ and $\f_i$, and $\pi_i:\f\ra \f_i$ will be the objects associated with $\R^3$.

\begin{definition}
A {\bf projective frame} is an indexed collection $\{W_1^i\}_{0\leq i\leq 3}\subset \f_1$ where any three elements span $\R^3$.  The {\bf standard projective frame} $\{\hat W_1^i\}_{0\leq i\leq 3}$  is given by $\hat W_1^i=\Span(e_i)$ for $1\leq i\leq 3$, and $\hat W_1^0=\Span(e_1+e_2+e_3)$.
\end{definition}
Note that for any projective frame $\{W_1^i\}_{0\leq i\leq 3}\subset \f_1$ there exists a unique $g\in \pgl(3,\R)$ such that $g(\hat W_1^i)=\hat W_1^i$ for $0\leq i\leq 3$, i.e. $\pgl(3,\R)$ acts freely transitively on the collection of projective frames.

The proof of Proposition~\ref{thm_fiber_preserving_projective} is inspired by the following elementary classical result:
\begin{theorem}[Fundamental Theorem of Projective Geometry]~
\label{thm_fundamental_theorem_projective_geometry}
Every fibration-preserving bijection $f:\f\ra\f$  agrees with some $g\in \pgl(3,\R)$.
\end{theorem}
Note that this is an equivalent reformulation of the classical result using the flag manifold rather the (traditional) projective plane.
\begin{proof}[Proof of Theorem~\ref{thm_fundamental_theorem_projective_geometry}]
Since  $f$ is a bijection, so are the maps  $f_i:\f_i\ra \f_i$ for $i\in \{1,2\}$.  It follows that $f_1$ maps $\pi_1(\pi_2^{-1}(W_2))$ bijectively to $\pi_1(\pi_2^{-1}(f_2(W_2)))$  for every plane $W_2\in\f_2$.  Therefore three elements of $\f_1$ lie in a plane if and only if their images under $f_1$ lie in a plane.  Hence $\{f_1(\hat W_1^i)\}_{0\leq i\leq 3}$ is a projective frame, so after composing $f$ with some element of $\pgl(3,\R)$, we may assume that $f_1(\hat W_1^i)=\hat W_1^i$ for $0\leq i\leq 3$.  It follows that $f_1$ also fixes $\hat W_1^4:=\hat W_2^{30}\cap \hat W_2^{12}$.  

We identify $\R^2$ with the affine plane $\{W_1\in \f_1\mid W_1\not\subset \hat W_2^{12}=\Span(e_1,e_2)\}$ by $(x_1,x_2)\leftrightarrow \Span(x_1,x_2,1)$.  We may define a bijection $\phi:\R^2\ra\R^2$ by $\Span(\phi(x_1,x_2),1)=f_1(\Span(x_1,x_2,1))$.  Using the fact that $f_1(\hat W_1^i)=\hat W_1^i$ for $0\leq i\leq 4$, one sees that:
\begin{enumerate}[label=(\alph*)]
\item $\phi$ fixes $(0,0)$, $(1,1)$.
\item $\phi$ maps lines bijectively to lines.
\item For every $v\in \{e_1,e_2,e_1+e_2\}$ and every line $L$ parallel to $v$,  the image $f(L)$ is a line parallel to $v$.
\end{enumerate}

Using a geometric construction  (see Figure~\ref{addition1})  it follows that the restriction of $\phi$ to  $\R\times\{0\}\simeq \R$ is a field isomorphism.  Since $\id_{\R}$ is the only automorphism of $\R$, it follows that $\phi$ fixes $\R\times\{0\}$.  Hence it also fixes $\{0\}\times\R$, then all of $\R^2$; and $f_1$ fixes $\f_1$, and $f$ fixes $\f$.

\begin{figure}
\centering
\includegraphics[width=90mm]{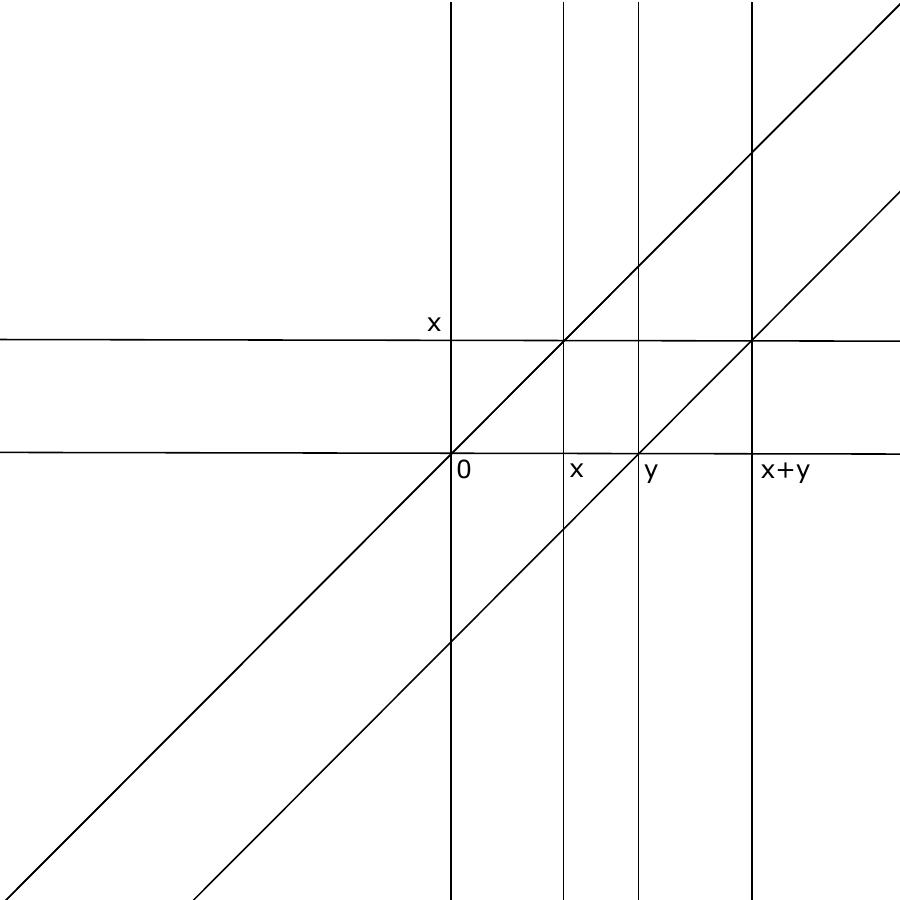}
\caption{ Geometric addition.
\label{addition1}}
\end{figure}

\end{proof}

\bigskip
The $n=3$ case of Proposition~\ref{thm_fiber_preserving_projective} may be viewed as a localized version of Theorem~\ref{thm_fundamental_theorem_projective_geometry} for mappings which are continuous, but not necessarily bijective. 

Before proceeding, we first give a rough idea how the argument goes.

The first step is to  show that local projectivity of $f$ propagates: if $f$ locally agrees with some $g\in \pgl(3,\R)$ near some point $x\in U$, then $f$ will locally agree with $g$ near points in the fibers of $\pi_1$, $\pi_2$ passing through $x$ (see Lemma~\ref{lem_projective_propagation}).  This readily implies that the subset of $U$ where $f$ agrees with $g$ is a connected component of $U$.  Thus we are reduced to finding a single point near which  $f$ is locally projective.  To that end, we consider $(W_1,W_2)\in U$ a point of differentiability where differential is an isomorphism.  Using the definition of differentiability, we argue that the map $f$ is nondegenerate near $(W_1,W_2)$, in the sense that the map $f_1$ induced on $\f_1$ carries a projective frame localized near $W_1$ to a projective frame (see below the claim in the proof of the $n=3$ case of Proposition~\ref{thm_fiber_preserving_projective}).  Then by pre/post-composing with suitable elements of $\pgl(3,\R)$, and working in $\f_1$, we are able to reduce to a map $\phi:\R^2\supset U\ra \R^2$ which satisfies a localized version of the conditions  (a)-(c) appearing in the proof of Theorem~\ref{thm_fundamental_theorem_projective_geometry}; this is shown in Lemma~\ref{lem_special_affine} to be locally affine, which then implies that $f$ is locally projective near $(W_1,W_2)$.

\bigskip
We now return to the preparations for the proof of Proposition~\ref{thm_fiber_preserving_projective}.

\bigskip
\begin{lemma}[Propagation of projectivity]
\label{lem_projective_propagation}
Suppose $U\subset \f$ is open and $f:U\ra \f$ is continuous and fibration preserving.
\ben
\item If $f$ agrees with $g\in \pgl(3,\R)$ near $(\bar W_1,\bar W_2)\in U$, then $f$ agrees with $g$ near the fiber $\pi_i^{-1}(\bar W_i)\cap U$, for $i\in \{1,2\}$.
\item For every $g\in \pgl(3,\R)$, the set where $f$ locally agrees with $g$ is a connected component of $U$.  \een
\end{lemma}
\begin{proof}
(1).  We prove the assertion for $i=2$; the case when $i=1$ is similar.  

After postcomposing with an element of $\pgl(3,\R)$, may assume without loss of generality that $f\equiv\id$ on an open set $V\subset U$ with $(\bar W_1,\bar W_2)\in V$.  

Pick $(W_1,W_2)\in U\cap \pi_2^{-1}(\bar W_2)$, i.e. $W_2=\bar W_2$.  Choose $(W_1,W_2')\in U\setminus \{(W_1,W_2)\}$ such that $\pi_2^{-1}(W_2')\cap V\neq \emptyset$.  

Now suppose $(W_1^j,W_2^j)\ra (W_1,W_2)$.  We may choose $\{W_2'^j\}\subset \f_2$ such that  $W_1^j\subset W_2'^j$ and $W_2'^j\ra W_2'$ as $j\ra\infty$.   Since $V$ is open, after dropping finitely many terms, we may assume that $\pi_2^{-1}(W_2^j)\cap V\neq \emptyset$ and $\pi^{-1}(W_2'^j)\cap V\neq\emptyset$.  Since $f$ is fibration-preserving and $f\equiv\id$ on $V$, it follows that $f(\pi_2^{-1}(W_2^j)\cap U)\subset \pi_2^{-1}(W_2^j)$ and $f(\pi^{-1}(W_2'^j)\cap U)\subset \pi^{-1}(W_2'^j)$. The map $f$ is fibration preserving, so $f(W_1^j,W_2^j), f(W_1^j,W_2'^j)\in \pi_1^{-1}(f_1(W_1^j))$.  Therefore $\pi_1^{-1}(f_1(W_1^j))$ intersects both $\pi_2^{-1}(W_2^j)$ and $\pi_2^{-1}(W_2'^j)$ nontrivially, forcing $f_1(W_1^j)=W_2^j\cap W_2'^j=W_1^j$.  Thus $f(W_1^j,W_2^j)=(W_1^j,W_2^j)$.  Since the sequence $\{(W_1^j,W_2^j)\}$ was arbitrary, this proves (1).

(2).   Define an equivalence relation on $U$ where $x,x'\in U$ are equivalent if there is a path $\ga:[0,1]\ra U$ from $x$ to $x'$ which is piecewise contained in a fiber of one of the fibrations $\pi_i:\f\ra \f_i$.  We claim that the equivalences classes are open subsets of $U$.  To see this, pick $x=g\cdot (W_1^-,W_2^-)\in U$, and note that the vector fields $X_{1,2}$, $X_{2,3}$ on $N$ are bracket generating and tangent to the cosets of $K_1$ and $K_2$, respectively;  and hence their pushforwards $X_{1,2}'$, $X_{2,3}'$ under the composition $N\stackrel{\al}{\ra}\hat N\stackrel{g}{\ra}g\cdot\hat N$ are bracket generating vector fields on $g\cdot\hat N$ which are tangent to the fibers of $\pi_1$ and $\pi_2$ respectively, by Lemma~\ref{lem_properties_fibration_between_flag_manifolds}(4).  By \cite[pp.50-52]{montgomery_book} 
 the equivalence class of $x$ contains a neighborhood of $x$.   It follows that the equivalence classes are connected components of $U$.  Since the set where $f$ locally agrees with $g$ is a union of equivalences classes by (1), assertion (2) follows.

\end{proof}

\begin{lemma}
\label{lem_special_affine}
Suppose $U\subset \R^2$ is a connected open subset, and $$\phi =(\phi_1,\phi_2):\R^2\supset U\ra\R^2$$ is a continuous map such for every $v\in \{e_1,e_2,e_1+e_2\}$ and every line $L$ parallel to $v$, the image of $L\cap U$ is contained in a line parallel to $v$.
Then $\phi$ is  of the form 
\begin{equation}
\label{eqn_special_affine}
\phi(x,y)=(mx+b_1,my+b_2)
\end{equation}
for some $m,b_1,b_2\in \R$.

\end{lemma}
\begin{proof}
First assume that $\phi$ is smooth.  
Applying the hypothesis with $v\in \{e_1,e_2\}$ implies that $\D_2\phi_1=\D_1\phi_2\equiv 0$, so $\phi_i$ depends only on $x_i$.  
Applying the hypothesis when $v=e_1+e_2$ we have
$$
0\equiv (\D_1+\D_2)(\phi_2-\phi_1)=\D_2\phi_2-\D_1\phi_1\,.
$$
But since $\D_i\phi_i$ depends only on $x_i$ this forces  $\D_1\phi_1=\D_2\phi_2=const$ and so \eqref{eqn_special_affine} holds.  Since the conditions are preserved by taking linear combinations and precomposing with  translations, the general case follows by mollification.

Alternatively, one may argue as follows.  Without loss of generality, one may assume that $\phi(0,0)=(0,0)$.  By geometric contruction, $\phi(x+x',0)=\phi(x,0)+\phi(x',0)$ when $x,x'\in (-r,r)$ for $r$ small.    Hence $\phi(x_1,0)=mx_1$ for some $m\in\R$, for $x_1\in (-r,r)$.  Invoking the hypotheses again, we get $\phi(x_1,x_2)=(mx_1,mx_2)$.  Thus the lemma holds locally, i.e.  for every $(x_1,x_2)\in U$ there is an open set $V_{x_1,x_2}$ containing $(x_1,x_2)$ and $m,b_1,b_2$ depending  on $(x_1,x_2)$ such that \eqref{eqn_special_affine} holds in $V_{x_1,x_2}$; since $U$ is connected and $m,b_1,b_2$ are locally constant, it follows that they are independent of $x_1,x_2$, and so \eqref{eqn_special_affine} holds.
\end{proof}

\bigskip\bigskip
\begin{definition}
An indexed tuple $\{W_1^i\}_{0\leq i\leq 4}\subset\f_1$ is an {\bf augmented projective frame} if $\{W_1^i\}_{0\leq i\leq 3}$ is a projective frame and 
$$
W_1^4=\Span(W_1^3,W_1^0)\cap\Span(W_1^1,W_1^2)\,.
$$   
The {\bf standard augmented projective frame} $\{\hat W_1^i\}_{0\leq i\leq 4}$ is given by 
\begin{equation}
\hat W_1^i=
\begin{cases}
e_1+e_2+e_3,\quad &i=0\\
e_i,\quad &1\leq i\leq 3\\
e_1+e_2,\quad &i=4
\end{cases}
\end{equation}
\end{definition}

Given a subset $\Si\subset \f_1$, we obtain (possibly empty) subsets of $\f_2$ and $\f$:
\begin{equation}
\label{eqn_f_sigma}
\begin{aligned}
\f_2(\Si)&:=\{\Span(\si_1,\si_2)\mid\;  \si_i\in\Si\,,\;\si_1\neq \si_2\}\,.\\
\f(\Si)&:=\{(W_1,W_2)\mid W_1\in\Si\,,\;W_2\in \f_2(\Si)\}\,.
\end{aligned}
\end{equation}

\begin{lemma}
\label{lem_small_projective_frame}
There is an augmented projective frame $\{\tilde W_1^i\}_{0\leq i\leq 4}\subset \f_1$ such that $\f(\{\tilde W_1^i\}_{0\leq i\leq 4})\subset \hat N$ and 
$$
(\tilde W_1^3,\Span(\tilde W_1^3,\tilde W_1^2))=(W_1^-,W_2^-)\,.
$$
\end{lemma}
\begin{proof}
Since $\hat N$ is an open dense subset of $\f$ by Lemma~\ref{lem_properties_of_f}, we may choose $g\in G$ such that $$g\cdot\f(\{\hat W_1^i\}_{0\leq i\leq 4})=\f(\{g\cdot\hat W_1^i\}_{0\leq i\leq 4})$$ lies in $\hat N$.   Then we may let $\tilde W_1^i:=(ng)\cdot \hat W_1^i$ for $0\leq i\leq 4$, where $n\in N$ satisfies $n\cdot (g\cdot(\hat W_1^3,\Span(\hat W_1^3,\hat W_1^2)))=(W_1^-,W_2^-)$.  
\end{proof}

\bigskip
\begin{proof}[Proof of Proposition~\ref{thm_fiber_preserving_projective} in the $n=3$ case.]
By Lemma~\ref{lem_projective_propagation}(2) and the connectedness of $U$, it suffices to show that $f$ locally agrees with some element of $\pgl(3,\R)$ near $x$.  Therefore after shrinking $U$ we may assume that $f$ is fibration-preserving, not just locally fibration-preserving.

After pre/postcomposing with elements of $\pgl(3,\R)$, we assume without loss of generality that $x=(W_1^-,W_2^-)\in U$, $f(x)=x$, and $D_Pf(x)$ is defined and an isomorphism. 

Let $\{\tilde W_1^i\}_{0\leq i\leq 4}\subset\hat N$ be the augmented projective frame from Lemma~\ref{lem_small_projective_frame}.  For $r>0$ let $\{W_1^i\}_{0\leq i\leq 4}$  be the image of $\{\tilde W_1^i\}_{0\leq i\leq 4}$ under $\hat\de_r$; we may assume $r$ is small enough that $\f(\{W_1^i\}_{0\leq i\leq 4})\in U$.

\begin{claim}
If $r$ is sufficiently small, then $\{f_1(W_1^i)\}_{0\leq i\leq 3}$ is a projective frame.
\end{claim}
\begin{proof}~
Since  $f$ is Pansu differentiable at $x$  and $D_Pf(x)$ is an isomorphism
\begin{equation}
\label{eqn_pansu_differential_f}
\hat \de_{r^{-1}}\circ f\circ \hat \de_r\ra D_Pf(x)
\end{equation}
uniformly on compact sets as $r\ra 0$, where $D_Pf:\hat N\ra \hat N$ is a graded automorphism; here we identify $N$ with $\hat N$.  
Note that  $D_Pf(x):\hat N\ra \hat N$ is fibration-preserving because it is a limit of fibration-preserving maps.  By Lemma~\ref{lem_properites_aut_gr_n}(3) there exist $\la_1,\la_2,\la_3\neq 0$ such that $\hat\Phi=\diag(\la_1,\la_2,\la_3) \in \pgl(3,\R)$  agrees with $D_Pf(x)$ on $\hat N$.  Therefore $(\hat\Phi)_1$ maps $\{\tilde W_1^i\}_{0\leq i\leq 3}$ to a projective frame.
It follows that for $r$ small  both  $\{(\hat \de_{r^{-1}}\circ f\circ \hat \de_r)_1(\tilde W_1^i)\}_{0\leq i\leq 3}$ and  $\{f_1(W_1^i)=(f\circ\hat\de_r)_1(\tilde W_1^i)\}_{0\leq i\leq 3}$ are projective frames.  
\end{proof}

\bigskip
Let $\{\hat W_1^i\}_{0\leq i\leq 4}$ be the standard augmented projective frame, and let $\hat W_2^{ij}:=\Span(\hat W_1^i,\hat W_1^j)$ for $0\leq i\neq j\leq 4$.
Since $\{W_1^i\}_{0\leq i\leq 3}$ and $\{f_1(W_1^i)\}_{0\leq i\leq 3}$ are both projective frames, there are elements $g_1,g_2\in \pgl(3,\R)$ such that $g_1(\hat W_1^i)=W_1^i$ and $g_2(f_1(W_1^i))=\hat W_1^i$ for $0\leq i\leq 3$.  We let $\hat U:=g_1^{-1}(U)\subset \f$ and $\hat f:=g_2\circ f\circ g_1:\hat U\ra \f$. 
    Then $\f(\{\hat W_1^i\}_{0\leq i\leq 4})\subset \hat U$, the map $\hat f$ is fibration-preserving, and  $(\hat f)_1(\hat W_1^i)=\hat W_1^i$ for all $0\leq i\leq 3$.  Since two distinct lines lie in a unique plane, and  two distinct planes intersect in a line, the fact that $\hat f$ is fibration-preserving implies that:
\ben
\item $\hat f_2(\hat W_2^{ij})=\hat W_2^{ij}$ for $0\leq i\neq j\leq 3$. 
\item $\hat f_1(\hat W_1^4)=(\hat f)_1(\hat W_2^{30}\cap \hat W_2^{12})=\hat W_1^4$.
\item $\hat f_2(\hat W_2^{ij})=\hat W_2^{ij}$ for $0\leq i\neq j\leq 4$.
\item $\hat f(\hat W_1^i,\hat W_2^{ij})=(\hat W_1^i,\hat W_2^{ij})$ for $0\leq i\neq j\leq 4$.
\een

Since $\hat f_1([e_3])=[e_3]$, 
for small $r>0$ we may define $\phi:\R^2\supset B(0,r)\ra \R^2$ by $\Span(\phi(x_1,x_2),1)=\hat f_1(\Span(x_1,x_2,1))$.  
  The map $\hat f$ is fibration-preserving and it fixes the standard augmented projective frame,  so the hypotheses of Lemma~\ref{lem_special_affine} hold for $\phi$.  Applying Lemma~\ref{lem_special_affine} to $\phi$, we get that for some $m\in \R$ we have $\hat f_1([x_1,x_2,1])=[mx_1,mx_2,1]$ for $x_1,x_2\in \R$ small.

Suppose $m=0$.  Then $\hat f$, and hence also $f$,  takes values in a single fiber of $\pi_1$ near $x=(W_1^-,W_2^-)$.  It follows that the Pansu differential $D_Pf(x)=\lim_{r\ra 0}\hat \de_{r^{-1}}\circ f\circ\de_r$ takes values in a single fiber of $\pi_1$; this contradicts the nondegeneracy of the Pansu differential.   Hence $m\neq 0$.

Let $\hat g:=\diag(m,m,1)\in\pgl(3,\R)$.  Then $\hat g_1:\f_1\ra\f_1$ agrees with $\hat f_1$ near $[e_3]$.  Since $\hat g$ and $\hat f$ are both fibration-preserving, they agree near $x=(W_1^-,W_2^-)$. 
\end{proof}

\bigskip\bigskip
\subsection{Proof of Proposition~\ref{thm_fiber_preserving_projective}, general case}
The $n\geq 4$ case is similar to the $n=3$ case.  The replacement for Lemma~\ref{lem_special_affine} is:

\begin{lemma}
\label{lem_special_affine_n_geq_4}
Suppose $V\subset \R^n$ is a connected open subset, and $$\phi=(\phi_1,\ldots,\phi_n):V\ra\R^n$$ is a continuous map such for every $v\in \{e_1,\ldots,e_n,e_1+\ldots+e_n\}$ and every line $L$ parallel to $v$, the image $\phi(L\cap V)$ is contained in a line parallel to $v$.
 Then $\phi$ is  of the form 
\begin{equation}
\label{eqn_special_affine_general_n}
\phi(x_1,\ldots,x_n)=(mx_1+b_1,\ldots,mx_n+b_n)
\end{equation}
for some $m,b_1,\ldots,b_n\in \R$.
\end{lemma}
We omit the proof as it is similar to the proof of Lemma~\ref{lem_special_affine}.

The $n\geq 4$ version of Lemma~\ref{lem_projective_propagation} is:
\begin{lemma}
\label{lem_projective_propagation_n_geq_4}
Suppose $U\subset \f$ is open and $f:U\ra \f$ is fibration preserving.
\ben
\item Suppose $f$ agrees with $g\in \pgl(n,\R)$ near $(\bar W_1,\ldots,\bar W_{n-1})\in U$.  For $i\in \{1, \cdots, n-1\}$, let $V_i$ be the connected component of $\pi_i^{-1}(\bar W_i)\cap U$ containing $(\bar W_1,\ldots,\bar W_{n-1})$.   Then $f$ agrees with $g$ near $V_i$.
\item For every $g\in \pgl(n,\R)$, the set where $f$ locally agrees with $g$ is a connected component of $U$.  \een
\end{lemma}
\begin{proof}
We prove the lemma by induction on the dimension $n\geq 3$. Lemma~\ref{lem_projective_propagation} covers the case $n=3$, so we may assume inductively that the lemma holds for dimensions strictly smaller than $n$.

(1).   We may assume without loss of generality that $g=\id$.  Suppose $f\equiv \id$  on an open subset $V\subset U$ containing $(\bar W_1,\ldots,\bar W_{n-1})$.   Arguing by contradiction, suppose  $i\in \{1,  \cdots, n-1\}$, and  for some $(W_1,\ldots,W_{n-1})\in V_i$, there is a sequence $ \{(W_1^j,\ldots,W_{n-1}^j)\}\subset U$ which converges to $(W_1,\ldots,W_{n-1})$ as $j\ra \infty$, but $$f(W_1^j,\ldots,W_{n-1}^j)\neq (W_1^j,\ldots,W_{n-1}^j)$$ for all $j$.  After passing to a subsequence, we may assume that for all $j$ the connected component of $\pi_i^{-1}(W_i^j)\cap U$ containing $(W_1^j,\ldots,W_{n-1}^j)$  intersects $V$. Since $f$ is fibration preserving and $f\equiv \id$ on $V$, it follows that $f$ maps $\pi_i^{-1}(W_i^j)\cap U$ into $\pi_i^{-1}(W_i^j)$.   Identifying $\pi_i^{-1}(W_i^j)$ with the flag manifold in $\R^{n-1}$, the restriction $f$ to $\pi_i^{-1}(W_i^j)$ induces a fibration-preserving mapping; by the induction assumption, since $f$ fixes $V$ it will also fix $(W_1^j,\ldots,W_{n-1}^j)$.  This is a contradiction.  Hence (1) holds.

(2).  Note that  each element of the basis $X_{1,2},\ldots,X_{n-1,n}$ for $V_1\subset\fn$ is tangent to one of the subgroups $K_j$ for $1\leq j\leq n-1$.  Hence we may use Lemma~\ref{lem_properties_fibration_between_flag_manifolds}(4) and argue as in Lemma~\ref{lem_projective_propagation}(2).  

\end{proof}

\bigskip

\begin{definition}
An indexed tuple $\{W_1^i\}_{0\leq i\leq n}\subset\f_1$ is a {\bf projective frame} if any subset of $n$ elements spans $\R^n$.  The {\bf standard projective frame}  $\{\hat W_1^i\}_{0\leq i\leq n}$ is given by $\hat W_1^i=e_i$ for $1\leq i\leq n$   
 and $\hat W_1^0=e_1+\ldots+e_n$.
An indexed tuple $\{W_1^i\}_{0\leq i\leq n+1}\subset\f_1$ is an {\bf augmented projective frame} if $\{W_1^i\}_{0\leq i\leq n}$ is a projective frame and $W_1^{n+1}=\Span(W_1^n,W_1^0)\cap\Span(W_1^1,\ldots,W_1^{n-1})$.   The {\bf standard augmented projective frame} is $\{\hat W_1^i\}_{0\leq i\leq n+1}$  with  $\hat W_1^{n+1}=e_1+\ldots+e_{n-1}$.
\end{definition}

Given a subset $\Si\subset \f_1$, we obtain (possibly empty) subsets of $\f_j$ and $\f$:
\begin{align*}
\f_j(\Si)&:=\{\Span(\Si')\mid\; \Si'\subset\Si,\; |\Si'|=j,\; \dim\Span(\Si')=j\}\,.\\
&\f(\Si):=\{(W_1,\ldots,W_{n-1})\in \f\mid W_j\in \f_j(\Si)\}\,.
\end{align*}

\begin{lemma}
There is an augmented projective frame $\{\tilde W_1^i\}_{0\leq i\leq n+1}\subset\f_1$ such that:
\bit
\item $\f(\{\tilde W_1^i\}_{0\leq i\leq n+1})$ is contained in $\hat N$.
\item  $\Span(\tilde W_1^n,\ldots,\tilde W_1^{n-j+1}) =\Span(e_n,\ldots,e_{n-j+1})$ for all $1\leq j\leq n-1$.
\eit 

\end{lemma}
\begin{proof}
This follows as in the proof of Lemma~\ref{lem_small_projective_frame}.
\end{proof}

\bigskip
\begin{proof}[Proof of Proposition~\ref{thm_fiber_preserving_projective}, $n\geq 4$ case]
The proof parallels the $n=3$    case closely, so we will be brief.

It suffices to show that $f$ locally agrees with some element of $\pgl(n,\R)$ near $x$. Also, we may assume without loss of generality that $f$ is fibration-preserving, $x=(W_1^-,\ldots,W_{n-1}^-)\in U$, $f(x)=x$, and that $D_Pf(x)$ is well-defined and an isomorphism.

For $r>0$ let $W_1^i:=\hat\de_r(\tilde W_1^i)$ for $0\leq i\leq n+1$; we take $r$ small enough that $\f(\{W_1^i\}_{0\leq i\leq n+1})\subset U$.

\begin{claim}  
For $r$ small $\{f_1(W_1^i)\}_{0\leq i\leq n}\subset \f_1$ is a projective frame.
\end{claim}
We omit the proof, as it is similar to claim in the proof of the $n=3$ case.

Let $g_1,g_2\in \pgl(n,\R)$ be such that $g_1(\hat W_1^i)=W_1^i$, $g_2(f_1(W_1^i))=\hat W_1^i$ for $0\leq i\leq n$.   We now define $\hat U=g_1^{-1}(U)$ and $\hat f:=g_2\circ f\circ g_1:\hat U\ra \f$.  

Arguing as in the $n=3$ case, one obtains that $\hat f$ is fibration-preserving,  $ \f(\{\hat W_1^i\}_{0\leq i\leq n+1})\subset\hat U$,
$f_j$ fixes  $\f_j(\{\hat W_1^i\}_{0\leq i\leq n+1})$ elementwise and $f$ fixes  $ \f(\{\hat W_1^i\}_{0\leq i\leq n+1})$ elementwise.

For $r>0$ small we define  $\phi:\R^n\supset B(0,r)\ra \R^n$ by 
$$
\Span(\phi(x_1,\ldots,n),1)=f_1(\Span(x_1,\ldots,x_n,1))\,.
$$
Applying Lemma~\ref{lem_special_affine_n_geq_4}, for some $m\in \R$ we get 
$$
\phi(x_1,\ldots,x_n)=(mx_1,\ldots,mx_n)\,.
$$  
As in the $m=3$ case we see that $m\neq 0$, and that $f_1$ agrees with $g:=\diag(m,\ldots,m,1)$ near $e_n$.  This implies that $f$ agrees with $g$ near $x=(W_1^-,\ldots,W_{n-1}^-)$.

\end{proof}

\bigskip
\section{The proof of Theorem~\ref{thm_main}}
\label{sec_proof_thm_main}

For every $x\in U$ choose a connected open set $U_x\subset U$ containing $x$ and group elements  $\bar g_x,\bar g_x'\in\pgl(n,\R)$ such that $\bar g_x(U_x),\bar g_x'(f(U_x))\subset \hat N$, and let  $\hat f_x:=\bar g_x'\circ f\circ \bar g_x^{-1}:\hat N\supset\hat V_x \ra\hat V_x'\subset \hat N$, where $\hat V_x:=\bar g_x(U_x)$,  $\hat V_x':=\bar g_x'(f(U_x))$.  Let $f_x:=\al^{-1}\circ  \hat f_x\circ\al:N\supset V_x\ra V_x'\subset N$ where $V_x:=\al^{-1}(\hat V_x)$, $V_x':=\al^{-1}(\hat V_x')$.     
    By Corollary~\ref{cor_preservation_coset_foliation}, for some $\eps_x\in \{0,1\}$, the map $\tau^{\eps_x}\circ f_x$ locally preserves the coset foliation of $K_j$ for all $1\leq j\leq n-1$.  Now Lemma~\ref{lem_properties_fibration_between_flag_manifolds}(4) gives that the map 
$$
\al\circ(\tau^{\eps_x}\circ f_x)\circ\al^{-1}=(\al\circ \tau^{\eps_x}\circ\al^{-1})\circ \hat f_x=\rho(\tau^{\eps_x})\circ \hat f_x
$$
locally preserves the fibration $\pi_j$ for $1\leq j\leq n-1$; here we have used the fact that $\rho(\tau)\circ\al=\al\circ \tau$.   Applying Proposition~\ref{thm_fiber_preserving_projective}, we see that $\rho(\tau^{\eps_x})\circ\hat f_x$ agrees with some element of $G$  and therefore $f=\rho(\Phi_x)\restr_{U_x}$ for some $\Phi_x\in \aut(G)$.    Since $\Phi_x$ is locally constant as a function of $x$,  by the connectedness of $U$, the automorphism $\Phi_x$ is independent of $x$, and so $f=\rho(\Phi)\restr_U$ for some $\Phi\in \aut(G)$.  By Lemma~\ref{lem_properties_aut_g}(1) we have $\Phi=\Phi_0^{\eps}\circ I_g$ for some $\eps\in \{0,1\}$, $g\in G$, where $\Phi_0$ denotes transpose-inverse; then  $\rho(\Phi)=(\rho(\Phi_0))^{\eps}\circ\rho(I_g)=\psi^\eps\circ g$ using Lemma~\ref{lem_action_aut_g_on_f}(5).

\bigskip
\section{The complex and quaternionic cases}~
\label{sec_complex_quaternionic}

The  arguments from the previous sections are also valid in the complex and quaternion cases, with   some straightforward modifications.  In this section we indicate what modifications are needed in these cases.  The necessity for these modifications  are due to the presence of nontrivial automorphisms of $\mathbb C$ and $\mathbb H$  and the non-commutativity of the  quaternions.

We first recall some facts about quaternions. Given any quaternion 
 $x=x_0+x_1i+x_2j+x_3k\in \mathbb H$ ($x_i\in \mathbb R$),  the conjugation of $x$ is
  $\bar x=x_0-x_1i-x_2j-x_3k$.   It is easy to check that $\overline{xy}=\bar y\bar x$ for any $x,y\in \mathbb H$.

   Let  $\lambda, \mu$  be  unit quaternions 
  satisfying $\lambda^2=\mu^2=(\lambda \mu)^2=-1$.    Set $\nu=\lambda\mu$. Then we have 
          $\mu=\nu\lambda$ and $\lambda=\mu\nu$. Define  a map 
           $h=h_{\lambda, \mu, \nu}: \mathbb H\rightarrow \mathbb H$ by
            $$h(a_0+ia_1+ja_2+ka_3)=a_0+\lambda a_1+\mu a_2+\nu a_3.$$ 
            Then it is easy to check that $h$ is an automorphism of $\mathbb H$: it is a real  linear isomorphism and $h(xy)=h(x)h(y)$ for any  $x, y\in \mathbb H$. Conversely, for any automorphism $h: \mathbb H\ra \mathbb H$,  if we set 
             $\lambda:=h(i)$, $\mu:=h(j)$, $\nu:=h(k)$, then  $\lambda^2=\mu^2=\nu^2=-1$,  $\nu=\lambda\mu$  and $h=h_{\lambda, \mu, \nu}$. 
     By the Skolem-Noether theorem, every automorphism     of  $\H$       is inner.

\bigskip
\subsection{Changes needed for Section 2}\label{changes2}

 Let $F=\mathbb C, \mathbb H$.   
  Let $\gl(n,F)$ be the group of invertible elements in the  ring $M_n(F)$ of $n\times n$ matrices with entries in $F$. 
  The objects   $P_F^+, P_F^-$ and $N_F$ are defined as before with $\mathbb R$ replaced by $F$.  The group $G_F$ is defined as before in the complex case with $\mathbb R$ replaced with $\mathbb C$:
 $$G_{\mathbb C}=\text{GL}(n, \mathbb C)/{\{\lambda\,\text{id}|0\not=\lambda\in \mathbb C\}}.$$
  In the quarternion case it is defined  by 
  $$G_{\mathbb H}=\text{GL}(n, \mathbb H)/{\{\lambda\,\text{id}\mid 0\not=\lambda\in \mathbb R\}}.$$
     Note that 
      $\{aI\mid a\in \mathbb H\backslash \{0\}\}$ is not normal in $GL(n, \mathbb H)$ and so we cannot quotient out by this subgroup.   
     Similarly,
     $P_{\mathbb C}=P_{\mathbb C}^-/{\{\lambda\,\text{id}\mid 0\not=\lambda\in \mathbb C\}} $  and 
   $P_{\mathbb H}=P_{\mathbb H}^-/{\{\lambda\,\text{id}\mid 0\not=\lambda\in \mathbb R\}} . $
  
\bigskip
\subsection*{The flag manifold}  
         We  view $F^n$ as a right $F$ module. 
        The {\bf flag manifold $\f_{F}$} is the set of (complete) flags in $F^n$, i.e. the collection of nested families of submodules of $F^n$
$$
W_1\subset\ldots \subset W_{n-1} 
$$
where $W_j$ has dimension (rank) $j$.  
  Matrix multiplication yields an action $\gl(n,F)\acts F^n$ by $F$-module automorphisms in the usual way,  which induces  actions $\gl(n,F)\acts \f_{k,F}$,
    where $ \f_{k,F}$  is the Grassmannian of submodules of  dimension (rank) $k$.

  Lemma \ref{lem_properties_of_f}   
   holds without changes.

\bigskip
\subsection*{Automorphisms}   The map $A\mapsto (A^*)^{-1}$ is a Lie group automorphism of $\gl(n,F)$, where $A^*$ denotes the conjugate transpose of $A$. 
  So the map    $\tau: GL(n, F)  \ra 
      GL(n, F)$    given by $\tau(A)=\Pi (A^*)^{-1}\Pi^{-1}$   is also an automorphism  and 
          induces an automorphism (still denoted by $\tau$)  of $\mathfrak{gl}(n, F)$ which is given by
        $$(\tau(A))_{ij}=-\overline{A}_{n-j+1, n-i+1}.$$


  For any automorphism $h$ of $F$, the automorphism 
        $ GL(n, F)   \ra   GL(n, F)$, $(a_{ij})\mapsto (h(a_{ij}))$  
        of $GL(n,  F)$ induces an automorphism of $G$, which we denote by 
         $\hat{h}$.

\begin{theorem}  (\cite[Theorems 1 and 2]{dieudonne}) \label{diedo}
Every automorphism of $G_F$ is induced by an automorphism of $GL(n, F)$.
 The group 
$\text{Aut}(G_F)$ is generated by $\tau$, maps of the form $\hat h$ (with 
$h\in Aut(F)$) and the inner automorphisms. 
\end{theorem}

      Notice that     the    automorphisms $h$  associated  with 
   Lie group automorphisms   are continuous. 
  We recall that there are only two continuous automorphism of   $\mathbb C$: the identity map and the complex conjugation.   On the other hand, 
by the Skolem-Noether theorem, every automorphism          $h: \mathbb H\ra \mathbb H$  is inner.
   It follows that the automorphism $\hat{h}$ of $G_{\mathbb H}$ 
         is also inner: if $h=I_a$ for some $a\in \mathbb H$, then $\hat{h}=I_g$ with
           $g=\text{diag}(a,  \cdots, a)$.
  


 The following is the counterpart of 
Lemma \ref{lem_properites_aut_gr_n} 
in the quaternion   and complex cases.

 
 \begin{lemma}\label{graded-quaternion}
Let $F=\mathbb C, \mathbb H$.   Let  $n\ge 3$  if $F=\mathbb H$  and   $n\ge 4$  if $F=\mathbb C$.  
 \ben
 \item If $\Phi\in \aut(G_F)$ and $\Phi(N_F)=N_F$, then $\Phi$ induces a graded automorphism of $N_F$ if and only if  $\Phi\restr_{N_F}=\tau^\eps\circ  \hat h\circ    I_g\restr_{N_F}$    for some $\eps\in \{0,1\}$, some  continuous  automorphism $h$ of $F$ and  $g=\diag(\la_1,\ldots,\la_n)\in G_F$.
\item Every graded automorphism of $N_F$ arises as in (1).
\item   For $1\leq j\leq n-1$ let $\fk_j\subset \fn_F$ be the Lie subalgebra generated by $\{a X_{i,i+1}|a\in F, i\neq n-j\}$, and $K_j\subset N_F$ be the Lie subgroup with Lie algebra $\fk_j$.    A graded automorphism $N_F\ra N_F$ is induced by conjugation by some $g=\diag(\la_1,\ldots,\la_n)$ if and only if   it  preserves the subgroups $K_j$ for $1\leq j\leq n-1$.
 \een
 \end{lemma}

We remark that Lemma \ref{graded-quaternion}  (2)   implies that  every graded automorphism of  ${\mathfrak n}_{n, \mathbb C}$ is either complex linear or complex  antilinear. 

\begin{proof}[Proof of Lemma \ref{graded-quaternion}  in the quaternion  case] 
  Let $n\ge 3$.  
 The proof of (1) and (3) are the same  (by using Theorem \ref{diedo})  
    as that of (1) and (3) in Lemma~\ref{lem_properites_aut_gr_n}.   Here we prove (2).  
 
   The proof is by induction on $n$.  We first consider the case $n=3$.
Denote  $X=X_{12}$,  $Y=X_{23}$  and $Z=X_{13}$.  
  We have $[aX, bY]=abZ$  for $a, b\in \mathbb H$.  Note that  the Lie bracket is not linear over $\mathbb H$.  
 Let $\mathbb H X=\{aX\mid a\in \mathbb H\}$ be the subspace spanned by $X$. It has dimension $4$ over $\mathbb R$. Similarly we have   $\mathbb H Y$ and  $\mathbb H Z$.
 The grading  $\mathfrak n_{3,\H}=V_1\oplus V_2$ is given by 
 $V_1=\mathbb H X\oplus \mathbb H Y$ and $V_2=\mathbb H Z$.

  Let $A:\mathfrak n_{3,\H}\rightarrow \mathfrak n_{3,\H}$ be a graded automorphism.
  We claim that   $A$    satisfies  either  $A(\mathbb H X)= \mathbb H X$,
$A(\mathbb H Y)= \mathbb H Y$ or $A(\mathbb H X)= \mathbb H Y$,
$A(\mathbb H Y)= \mathbb H X$.  
 There are $a, b\in \mathbb H$  such that
 $A(X)=aX+bY$.  
 To prove the claim it suffices to show that  $a=0$ or $b=0$.  Suppose $a, b\not=0$ we shall get a contradiction.   
There are $c, d\in \mathbb H$ such that $A(iX)=cX+dY$.
As $[X, iX]=0$ we have $0=[A(X), A(iX)]=(ad-cb)Z$ (being careful about the order in  $cb$), which yields  $a^{-1}c=db^{-1}$.
 Set $\lambda=a^{-1}c$. We have $c=a\lambda$ and $d=\lambda b$  and so
 $A(iX)=a\lambda X+\lambda b Y$. Similarly there are $\mu, \nu\in \mathbb H$ such that
  $A(jX)=a\mu X+\mu b Y$ and  $A(kX)=a\nu X+\nu b Y$. By further considering
  the brackets between $A(iX)$, $A(jX)$, $A(kX)$ we get
   $a\lambda \mu b=a\mu \lambda b$,   $a\lambda \nu b=a\nu \lambda b$  and 
    $a \mu \nu b=a\nu \mu  b$.  It follows that $\lambda$, $\mu$ and $\nu$ commute with each other. Recall the fact that two quaternions commute with each  other  if and only   if their imaginary parts are real multiples of each other. Hence there are real numbers
     $r_i, t_i$   ($i=1, 2,3$) and a purely imaginary  quaternion $h$ such that
      $\lambda=r_1+t_1 h$,   $\mu=r_2+t_2 h$,   $\nu=r_3+t_3 h$.     We then have 
       $A(iX)=r_1(aX+bY)+t_1(ahX+hbY)$ and so $A(iX)$ lies in the $2$-dimensional real vector subspace spanned by $aX+bY$ and $ahX+hbY$. Similarly $A(jX)$ and $A(kX)$ also lie in  this subspace, contradicting the   fact that $A$ is an isomorphism. This finishes the proof of the claim.

After possibly composing  $A$ with $\tau$ we may assume that $A(\mathbb H X)=\mathbb H X$ and 
 $A(\mathbb H Y)=\mathbb H Y$.  
There are  $a_0, a_1, a_2, a_3, b_0, b_1, b_2, b_3\in \mathbb H$   such that
 $A(X)=a_0X$, $A(iX)=a_1X$,  $A(jX)=a_2X$,  $A(kX)=a_3X$  and 
 $A(Y)=b_0Y$, $A(iY)=b_1Y$,  $A(jY)=b_2Y$,  $A(kY)=b_3Y$.   By applying $A$ to 
 $[X, iY]=[iX, Y]=[jX, kY]=-[kX, jY]$,  
   $[X, jY]=[jX, Y]=[kX, iY]=-[iX, kY]$  and $[X, kY]=[kX, Y]=[iX, jY]=-[jX, iY]$ we obtain
   $$a_0b_1=a_1b_0=a_2b_3=-a_3b_2,$$
    $$a_0b_2=a_2b_0=a_3b_1=-a_1b_3,$$
    $$a_0b_3=a_3b_0=a_1b_2=-a_2b_1.$$
      From these  we get
       $$a_0^{-1}a_1=b_1b_0^{-1}=-b_2b_3^{-1}=b_3b_2^{-1},$$
        $$a_0^{-1}a_2=b_1b_3^{-1}=b_2b_0^{-1}=-b_3b_1^{-1},$$
        $$a_0^{-1}a_3=-b_1b_2^{-1}=b_2b_1^{-1}=b_3b_0^{-1}.$$
        Set $\lambda=a_0^{-1}a_1$, $\mu=a_0^{-1}a_2$  and $\nu=a_0^{-1}a_3$.
          From $\lambda=-b_2b_3^{-1}=b_3b_2^{-1},$  we get $\lambda^2=-1$. Similarly from $\mu=b_1b_3^{-1}=-b_3b_1^{-1},$  we get $\mu^2=-1$.  Finally 
           $\nu=b_3b_0^{-1}=b_3b_2^{-1} b_2b_0^{-1}=\lambda\mu$.  
        Also notice $a_1=a_0\lambda$, $a_2=a_0\mu$, $a_3=a_0\nu$ and 
         $b_1=\lambda b_0$, $b_2=\mu b_0$, $b_3=\nu b_0$. So the automorphism $A$ is given by $A(X)=a_0X$,  $A(iX)=a_0\lambda X$, $A(jX)=a_0\mu X$, $A(kX)=a_0\nu X$ and
       $A(Y)=b_0Y$,  $A(iY)=\lambda b_0Y$, $A(jY)=\mu b_0 Y$, $A(kY)=\nu b_0 Y$.  
       Now it is easy to check that   $A=\text{Ad}_g\restr_{\mathfrak n_\H}\circ \hat{h}_{\lambda, \mu,\nu}$, where  $g=\text{diag}(a_0, 1, b_0^{-1})$.

       Now assume $n\ge 4$. 
 By an argument similar to the real case  
 (using rank  and the induction hypothesis), we  get  (after possibly composing with $\tau$)  $A(\mathbb H X_{i, i+1})=\mathbb H X_{i, i+1}$  for each $1\le i\le n-1$.    Then there are nonzero quaternions $a_1, \cdots, a_{n-1}$ such that 
  $A(X_{i, i+1})=a_iX_{i, i+1}$.    Set $b_n=1$ and  $b_i=(a_i\cdots a_{n-1})^{-1}$   for $1\le i\le n-1$,    and   let 
   $g=\text{diag}(b_1, \cdots,  b_n)$.   
  By composing  $A$  with  $\text{Ad}_g\restr_{\mathfrak n_\H}$   we may assume that 
   $A(X_{i,i+1})=X_{i, i+1}$  for each $1\le i\le n-1$.    
  Now for each $1\le i\le n-2$,   $\mathbb H X_{i, i+1}$ and $\mathbb H X_{i+1, i+2}$ generate a  Lie subalgebra isomorphic to $\mathfrak{n}_3$. By the previous paragraph we see that there are $\lambda_i, \mu_i$ satisfying $\lambda_i^2=\mu_i^2=(\lambda_i\mu_i)^2=-1$   such that 
   $A((a_0+a_1i+a_2j+a_3k)X)=(a_0+a_1\lambda_i+a_2\mu_i+a_3\nu_i)X$ for $X=X_{i,i+1}, X_{i+1, i+2}, X_{i, i+2}$, where $\nu_i=\lambda_i\mu_i$  and $a_0, a_1, a_2, a_3\in \mathbb R$.   
     By considering the Lie subalgebra generated by $X_{i+1, i+2}$ and 
    $X_{i+2, i+3}$  and comparing the values for
    $A((a_0+a_1i+a_2j+a_3k)X_{i+1,  i+2})$  
      we get  
       $\lambda_i=\lambda_{i+1}$, $\mu_i=\mu_{i+1}$ and $\nu_i=\nu_{i+1}$.  It follows that 
      $A=\hat{h}_{\lambda, \mu, \nu}$ with $\lambda=\lambda_1$, $\mu=\mu_1$ and $\nu=\lambda\mu$.  
\end{proof}

\bigskip
\begin{proof}[Proof of Lemma \ref{graded-quaternion}  in the   complex  case]
  Let $n\ge 4$.  
 The proofs of (1)  and (3)  are the  same as in the real case  by using Theoem \ref{diedo}.  As remarked above,   the automorphism $\hat h$    is either the identity map or the complex conjugation.   
 
 (2)   The proof is by induction on $n$.   We first consider the case $n=4$.   
 Let $A: {\mathfrak n}_{4, \mathbb C} \ra  {\mathfrak n}_{4, \mathbb C}$ be  a 
    graded automorphism. We  observe that ${\mathfrak n}_{n, \mathbb C}$ is the complexification of $\mathfrak n_n$.   An easy calculation shows 
     $\text{rank}(\text{ad}\, x)=\dim(\text{ad}\, x( {\mathfrak n}_{4, \mathbb C}))\ge 4$ for any   nonzero element $x$  in the first layer of $ {\mathfrak n}_{4, \mathbb C} $.  Clearly,
          $\text{rank}(\text{ad}\, x)\le 6$.  So the condition  
          $$\max\{\text{rank}(\text{ad}\, x)|x\in V_1\}<2 \min\{\text{rank}(\text{ad}\, x)|0\not=x\in V_1\}$$  in    \cite[Lemma 4.7]{KMX1}   is satisfied   and we conclude that 
            every graded automorphism of  ${\mathfrak n}_{4, \mathbb C}$ is either complex linear or complex  antilinear.   So after possibly composing $A$ with the complex conjugation we may assume $A$ is complex linear.  The rest of the arguemnt in the case $n=4$ is the same as in the real case. 
            
        Now let $n\ge 5$ and assume the statement holds for      all integers less than $n$.
         Let  $A: {\mathfrak n}_{n, \mathbb C} \ra  {\mathfrak n}_{n, \mathbb C}$ be  a 
    graded automorphism. 
        Arguing as in the real case using rank, we see that after   possibly composing 
    $A$ with      $\tau$     we have $A(\mathbb C X_{i, i+1})=\mathbb C X_{i, i+1}$  for each $i$.
        So $A (\mathfrak n_+)=\mathfrak n_+$  and $A (\mathfrak n_-)=\mathfrak n_-$,
         where $\mathfrak n_+$ is the Lie sub-algebra of $ {\mathfrak n}_{n, \mathbb C}$ generated by $\{X_{i, i+1}, 1\le i\le n-2\}$ and 
        $\mathfrak n_-$ is the Lie sub-algebra of $ {\mathfrak n}_{n, \mathbb C}$ generated by $\{X_{i, i+1}, 2\le i\le n-1\}$.  Since 
        $\mathfrak n_+$  and $\mathfrak n_-$  are isomorphic to 
        ${\mathfrak n}_{n-1, \mathbb C} $, the induction hypothesis implies that   each of 
         $A|_{\mathfrak n_+}$,  $A|_{\mathfrak n_+}$ is either complex linear or complex antilinear.  Since  $\mathfrak n_+$  and $\mathfrak n_-$   have nontrivial intersection,
           either both  $A|_{\mathfrak n_+}$,  $A|_{\mathfrak n_+}$ are  complex linear or both are complex antilinear.   Hence after possibly composing $A$ with the complex conjugation we   may assume $A$ is complex linear.   As   $A$ also satisfies 
            $A(\mathbb C X_{i, i+1})=\mathbb C X_{i, i+1}$  for all $i$, we see that $A=\text{Ad}_g$ for some $g=\text{diag}(\lambda_1, \cdots, \lambda_n)\in G_{\mathbb C}$. 
             This finishes the proof of (2).  
\end{proof}

\bigskip
The ``Hermitian product'' on $\mathbb H^n$  is defined by 
    $<z,w>=\sum_i\bar{z_i}w_i$ for $z=(z_i), w=(w_i)\in \mathbb H^n$.  
Then one can check by direct calculation that  $\overline{<z,w>}=<w,z>$ and $<A^*z, w>=<z, Aw>$ for $z,w\in \mathbb H^n$ and $A\in M_n(\mathbb H)$.  
For any $\mathbb H$-linear subspace (submodule)  $W$ of  $\mathbb H^n$,  the ``orthogonal complement'' $W^\perp$ is defined by $W^\perp=\{z\in \mathbb H^n| <z,w>=0\; \forall w\in W\}$.

Lemma \ref{lem_action_aut_g_on_f}     holds for $F=\mathbb C, \mathbb H$ where in (5) the  automorphism 
  $\Phi_0$ is induced  by  the map $A\mapsto (A^*)^{-1}$.   
  The proof of Lemma \ref{lem_action_aut_g_on_f} (5) goes through in the quaternion case since
$$<(g^{-1})^*W^+_j, gW^-_{n-j}>=<W^+_j, g^{-1}gW^-_{n-j}>=0.$$
Of course, the proof is also valid in the complex case if we use the standard Hermitian product in $\mathbb C^n$.   

Lemmas~\ref{lem_properties_fibration_between_flag_manifolds} and \ref{lem_dilation_dynamics}     hold   in the complex and quaternion cases without change.

\bigskip
\subsection{Changes needed for Section 3}\label{changes3}
 
 Lemma~\ref{lem_main_no_oscillation}    and 
  Corollary   \ref{cor_preservation_coset_foliation}
  hold   in the complex case for $n\ge 4$ and in the quaternion case for $n\ge 3$.  Note that the analog of Lemma~\ref{lem_main_no_oscillation}  in the $F=\C$ case fails when $n=3$, as in the real $n=3$ case.  We indicate the changes needed in the quaternion case below.  We skip the complex case since it is similar. Alternatively the complex case also follows from  Corollary 8.2 of  \cite{KMX1} as  by Lemma 
  \ref{graded-quaternion} (2)   
   every graded automorphism of  $\mathfrak n_{n, \mathbb C}$ is either complex linear or complex antilinear.  
 
 \begin{lemma}
\label{no_oscillation_quarternion}
Let $n\ge 3$. 
Let $U\subset N_{\mathbb H}$ be a connected open subset, and for some $p>\nu$ let $f:N_{\mathbb H}\supset U\ra N_{\mathbb H}$ be a $W^{1,p}_{\loc}$-mapping whose Pansu differential is an isomorphism almost everywhere.  Then after possibly composing with $\tau$, if necessary, 
for a.e. $x\in U$ the Pansu differential $D_Pf(x)$ preserves the subspace $\H X_{i,i+1}\subset V_1$ for every $1\leq i\leq n-1$.
\end{lemma}

 

                 Here we indicate the differential forms used in the calculations, the rest of the argument being the same. 
             For $1\le s<t\le n$,        denote $Y_{st}=iX_{st}$, $Z_{st}=jX_{st}$,
               $W_{st}=kX_{st}$.   
                 Then $\{X_{st},   Y_{st},   Z_{st}, W_{st} |1\le s< t\le n\}$  form a basis of left invariant vector fields on $N$.   The only nontrivial bracket relations between the basis elements are given by (for $1\le s_1<s_2<s_3\le n$):
                 $$-[X_{s_1 s_2}, X_{s_2 s_3}]=[Y_{s_1 s_2}, Y_{s_2 s_3}]=[Z_{s_1 s_2}, Z_{s_2 s_3}]=[W_{s_1 s_2}, W_{s_2 s_3}]=-X_{s_1s_3}$$
                 $$[X_{s_1 s_2}, Y_{s_2 s_3}]=[Y_{s_1 s_2}, X_{s_2 s_3}]=[Z_{s_1 s_2}, W_{s_2 s_3}]=-[W_{s_1 s_2}, Z_{s_2 s_3}]=Y_{s_1s_3}$$
                 $$[X_{s_1 s_2}, Z_{s_2 s_3}]=[Z_{s_1 s_2}, X_{s_2 s_3}]=[W_{s_1 s_2}, Y_{s_2 s_3}]=-[Y_{s_1 s_2}, W_{s_2 s_3}]=Z_{s_1s_3}$$
                 $$[X_{s_1 s_2}, W_{s_2 s_3}]=[W_{s_1 s_2}, X_{s_2 s_3}]=[Y_{s_1 s_2}, Z_{s_2 s_3}]=-[Z_{s_1 s_2}, Y_{s_2 s_3}]=W_{s_1s_3}. $$
             Let $\alpha_{st}$, $\beta_{st}$, $\gamma_{st}$, $\eta_{st}$  be the dual basis of left invariant  $1$-forms.  We have 
             $$d\alpha_{s_1s_3}=\sum_{s_1<s_2<s_3}   (-\alpha_{s_1s_2}\wedge\alpha_{s_2s_3}+\beta_{s_1s_2}\wedge \beta_{s_2s_3}+\gamma_{s_1s_2}\wedge \gamma_{s_2s_3}+\eta_{s_1s_2}\wedge\eta_{s_2s_3} )  $$
              $$d\beta_{s_1s_3}=\sum_{s_1<s_2<s_3}   (-\alpha_{s_1s_2}\wedge\beta_{s_2s_3}-\beta_{s_1s_2}\wedge \alpha_{s_2s_3}-\gamma_{s_1s_2}\wedge \eta_{s_2s_3}+\eta_{s_1s_2}\wedge\gamma_{s_2s_3} )  $$
              $$d\gamma_{s_1s_3}=\sum_{s_1<s_2<s_3}   (-\alpha_{s_1s_2}\wedge\gamma_{s_2s_3}-\gamma_{s_1s_2}\wedge \alpha_{s_2s_3}-\eta_{s_1s_2}\wedge \beta_{s_2s_3}+\beta_{s_1s_2}\wedge\eta_{s_2s_3} )  $$
              $$d\eta_{s_1s_3}=\sum_{s_1<s_2<s_3}   (-\alpha_{s_1s_2}\wedge\eta_{s_2s_3}-\eta_{s_1s_2}\wedge \alpha_{s_2s_3}-\beta_{s_1s_2}\wedge \gamma_{s_2s_3}+\gamma_{s_1s_2}\wedge\beta_{s_2s_3} ).  $$
             We pull back the following closed left invariant forms
             $$\omega_+=\bigwedge_{2\le s\le n} (\alpha_{1s}\wedge\beta_{1s}\wedge\gamma_{1s}\wedge\eta_{1s})$$
             $$\omega_-=\bigwedge_{1\le s\le n-1} (\alpha_{sn}\wedge\beta_{sn}\wedge\gamma_{sn}\wedge\eta_{sn}).$$
             By Lemma \ref{graded-quaternion} on graded automorphisms,  the pull-back has the form $f_P^*\omega_+=u_+\omega_++u_-\omega_-$ as before. 
              Let 
              $$\eta_-=\bigwedge_{2\le s<t\le n} (\alpha_{st}\wedge\beta_{st}\wedge
              \gamma_{st}\wedge\eta_{st})$$
              $$\eta_+=\bigwedge_{1\le s<t\le n-1} (\alpha_{st}\wedge\beta_{st}\wedge
              \gamma_{st}\wedge\eta_{st}).$$
             We apply the  pull-back theorem  to 
             $f_P^*\omega_+$ and $\eta=i_X\eta_-$,  $i_X\eta_+$, where 
             $X\in\{X_{s (s+1)}, Y_{s(s+1)},  Z_{s(s+1)},  W_{s(s+1)}\}$ and $i_X$ denotes the interior product with respect to $X$.  
              As before, this yields $Xu_+=0$ for 
              $X\in\{X_{s (s+1)}, Y_{s(s+1)},  Z_{s(s+1)},  W_{s(s+1)}\}$ with $2\le s\le n-1$   and $Xu_-=0$ for 
              $X\in\{X_{s (s+1)}, Y_{s(s+1)},  Z_{s(s+1)},  W_{s(s+1)}\}$ with $1\le s\le n-2$.  The rest of the argument is the same as in the real case.


\bigskip
\subsection{Changes needed for Section 4}\label{changes4}         
                 

In the quaternion case, we note that the action of  $PGL(n, \mathbb H)$   on the projective frames is still transitive, but is no longer free.  The reason is that   $a I_n$ defines a nontrivial element in 
$PGL(n, \mathbb H)$  for $a\in \mathbb H\setminus \mathbb R$,  but fixes the standard projective frame.

Below is a version of Lemma \ref{lem_special_affine} for the quaternion case. 
   A similar statement holds for the complex case.   A line in $\mathbb H^2$ is a subset of the form $\{(a_1, a_2)x+(b_1, b_2)|x\in \mathbb H\}$ for some $(b_1, b_2)\in \mathbb H^2$,
    $(0,0)\not=(a_1, a_2)\in \mathbb H^2$.  This line is said to be parallel to $(a_1, a_2)=a_1e_1+a_2e_2$.

  \begin{lemma}
\label{lem_special_affine.quaternion}
Suppose $U\subset \mathbb H^2$ is a connected open subset, and $$\phi =(\phi_1,\phi_2):\mathbb H^2\supset U\ra\mathbb H^2$$ is a continuous map such that for every $v\in \{e_1,e_2,e_1+e_2\}$ and every line $L$ parallel to $v$, the image of $L\cap U$ is contained in a line parallel to $v$.   Assume further that for each $q\in \{i, j, k\}$, lines parallel to 
$e_1+qe_2$ are mapped into   lines (not necessarily parallel to  $e_1+qe_2$).   
Then $\phi$ is  of the form 
\begin{equation}
\phi(x,y)=(a h(x)+b_1, a  h(y)+b_2)
\end{equation}
  where  $a, b_1,b_2\in \mathbb H$  and $h: \mathbb H\ra \mathbb H$  is either an automorphism of $\mathbb H$ or  the zero map.

\end{lemma}

\begin{proof}
  The argument of Lemma \ref{lem_special_affine}  yields that 
  $\phi$ is  of the form 
  $\phi(x,y)=(m(x)+b_1, m(y)+b_2)$,  where 
 $b_1,b_2\in \mathbb H$  and $m:\mathbb H\ra \mathbb H$ is a real linear map.
  We shall show that  either $m$ is the zero map or there is some automorphism $h$ of $\mathbb H$ and some
 $a\in \mathbb H$ such that $m(x)=a h(x)$.   We may assume $b_1=b_2=0$ after possibly composing $\phi$ with a translation.  Then the assumption implies that the line $(1,i)\mathbb H$ is mapped 
 by $\phi$ into a line $(a_1, a_2)\mathbb H$ (at least one of $a_1, a_2$ is nonzero) through the origin.  There are $t_1, t_2\in \mathbb H$ such that 
  $(m(1), m(i))=\phi(1,i)=(a_1t_1, a_2t_1)$ and $(m(i), m(-1))=\phi(i, -1)=(a_1t_2, a_2t_2)$.  By comparing the components we get $m(i)=a_2t_1=a_1t_2$, $m(1)=a_1t_1=-a_2t_2$.  Suppose $m(1)=0$. As  at least one of $a_1, a_2$ is nonzero we have  $t_1=0$ or $t_2=0$, which implies $m(i)=0$. Similarly $m(j)=m(k)=0$. In this case $m$ is the zero map. 
 
 Now we assume $m(1)\not=0$.  
  Then    there is some $\lambda\in \mathbb H$ such that
  $\phi(i, -1)=\phi(1,i)\lambda$. By comparing the two components of both sides we get
   $m(i)=m(1)\lambda$ and $-m(1)=m(-1)=m(i)\lambda$, which yields
    $\lambda^2=-1$.   Similarly by considering the lines $(1,j)\mathbb H$ and 
    $(1,k)\mathbb H$  we  see that there are $\mu$ and $\nu$ satisfying $\mu^2=\nu^2=-1$ such that  $m(j)=m(1)\mu$,   $-m(1)=m(-1)=m(j)\mu$, 
     $m(k)=m(1)\nu$ and $-m(1)=m(-1)=m(k)\nu$.

     Since $(j,k)=(1,i)j\in (1,i)\mathbb H$, there is some $c\in \mathbb H$ such that
   $\phi(j,k)=\phi(1,i) c$.   This gives us $m(j)= m(1) c$ and $m(k)=m(i)c$. As we also have  $m(i)=m(1)\lambda$,  
    $m(j)=m(1)\mu$  and $m(k)=m(1)\nu$,  we conclude $c=\mu$ and $\nu=\lambda \mu$.    The three numbers $\lambda, \mu,\nu$ satisfy $\lambda^2=\mu^2=\nu^2=-1$ and $\nu=\lambda\mu$.
     As $m$ is $\mathbb R$-linear, for any $x=x_0+ix_1+jx_2+kx_3$ ($x_i\in \mathbb R$) we get $m(x)=m(1) h_{\lambda, \mu, \nu}(x)$.

   \end{proof}

     The arguments in Section \ref{sec_rigidity_foliation_preserving_maps} show that we may assume  the map $\phi$  sends 
      lines parallel to $v\in \{e_1, e_2, e_1+e_2\}$ 
        into lines parallel to $v$.  We claim that for any $q\in \{i,j,k\}$, lines parallel to $e_1+qe_2$ are mapped by $\phi$ into a family of parallel lines (not necessarily parallel to  $e_1+qe_2$).  To see this, we notice that for a suitable diagonal matrix $g=\text{diag}(a_1, a_2, 1)\in GL(3, \mathbb H)$, $g\circ \hat{f}_1$ satisfy $g\circ \hat{f}_1(\text{span}(e_1+qe_2+e_3))=\text{span}(e_1+qe_2+e_3)$  
     and $g\circ \hat{f}_1(\hat{W}^i_1)=\hat{W}^i_1$ for $i=1,2,3$. 
     Since  $\H^3$ is a right $\H$ module, here 
      for any $(x_1, x_2, x_3)\in \H^3$,  $\Span(x_1,x_2, x_3)=\{(x_1x, x_2x, x_3x)|x\in \H\}$.  
     Then  it follows that
    $\bar g\circ \phi$ sends lines parallel to $e_1+qe_2$ to lines parallel to 
$e_1+qe_2$, where $\bar{g}:\mathbb H^2\ra \mathbb H^2$  is the linear map given  by the diagonal matrix $\text{diag}(a_1, a_2)$. 
  Consequently $\phi$ sends lines parallel to $e_1+qe_2$ into a family of parallel lines.

\bigskip
\begin{proof}[Proof  of   Counterpart of Lemma \ref{thm_fiber_preserving_projective}   in the case $n=3$, $F=\mathbb H$].   
    Only the last paragraph   and the third last paragraph of the proof of Lemma \ref{thm_fiber_preserving_projective}
     need some changes.  
      As before we have  $\hat f_1([e_3])=[e_3]$. Hence  
for small $r>0$ we may define $\phi:\H^2\supset B(0,r)\ra \H^2$ by $\Span(\phi(x_1,x_2),1)=\hat f_1(\Span(x_1,x_2,1))$. The fact
  $\hat f_1([e_3])=[e_3]$  implies $\phi(0,0)=(0,0)$.   

      By Lemma \ref{lem_special_affine.quaternion}
     the map $\phi: \mathbb H^2\supset B(0,r) \ra \mathbb H^2$  in this case has the form
$\phi(x_1,  x_2)=(a h(x_1)+b_1, a  h(x_2)+b_2)$
  where  $a, b_1,b_2\in \mathbb H$  and $h: \mathbb H\ra \mathbb H$  is either an automorphism of $\mathbb H$ or  the zero map. 
   As $\phi(0,0)=(0,0)$, we have  $b_1=b_2=0$ and so 
     $\phi(x_1,  x_2)=(a h(x_1), a  h(x_2))$.   
   The argument for $m\not=0$ applies here and shows that $h$ is an automorphism.  Since any automorphism of $H$ is inner, 
   $h(x)=bxb^{-1}$   for some $0\not= b\in \mathbb H$  and so 
   $\phi(x_1,  x_2)=(a bx_1b^{-1}, a bx_2b^{-1})$.
Let $\hat g:=\diag(ab,  ab,  b)\in\pgl(3,\H)$.  Then $\hat g_1:\f_1\ra\f_1$ agrees with $\hat f_1$ near $[e_3]$.  Since $\hat g$ and $\hat f$ are both fibration-preserving, they agree near $x=(W_1^-,W_2^-)$. 
\end{proof}

\bigskip
\subsection{Changes needed for Section 5}\label{changes5}  

The counterpart of Theorem \ref{thm_main}  for the complex and quaternion cases is: 

\begin{theorem}\label{thm_mainF}
  Let $U\subset \f_F$ be a connected open subset, and $f:U\ra \f_F$ be a $W^{1,p}_{\loc}$-mapping for $p>\nu$, such that the Pansu differential is an isomorphism almost everywhere. 
\ben
\item   If $F=\mathbb C$ and $n\ge 4$, then   $f$ is the restriction of a diffeomorphism $\f\ra \f$ of the form $\psi^{\eps_1}\circ  \mathcal{C}^{\eps_2} \circ g$  where $g\in \pgl(n,\C)$, $\eps_i\in \{0,1\}$ and $\mathcal C$ is complex conjugation. \newline
\item    If $F=\mathbb H$ and $n\ge 3$, then   $f$ is the restriction of a diffeomorphism $\f\ra \f$ of the form $\psi^{\eps} \circ g$ where $g\in \pgl(n,\H)$, $\eps\in \{0,1\}$.
\een
\end{theorem}

The proof of Theorem \ref{thm_mainF} is the same as that of Theorem \ref{thm_main} except in the complex case  where we may need to   compose $f$ with the complex conjugation if necessary.

\bigskip
\section{Global quasiconformal homeomorphisms}
\label{sec_global_qc_homeos}

In this section we identify all global nondegenerate Sobolev maps  $N\ra N$.  These are exactly the graded affine maps of $N$.  This result  is an immediate consequence of  
 Theorem \ref{thm_main}.

An affine map of a Lie group  $G$  is a   map of the form $L_g\circ \phi$, where $\phi$ is an automorphism of $G$ and $L_g$ is left translation by $g\in G$.   A graded affine map of a Carnot group $N$   is an affine map where the automorphism is a graded automorphism of  $N$.  

The following result  applies to global quasiconformal homeomorphisms since quasiconformal maps are nondegenerate Sobolev maps.   

  \begin{theorem}\label{rigidity global qc}~
   Let $N$ be the Iwasawa group of  $\text{GL}(n, F)$,  with $n\ge 4$ for $F=\mathbb R, \mathbb C$ and $n\ge 3$ for $F=\mathbb H$.   
   Suppose  $f:N\ra N$ is  a $W^{1,p}_{\loc}$-mapping for $p>\nu$, such that the Pansu differential is an isomorphism almost everywhere.  Then $f$ is  a graded affine map of $N$.
         \end{theorem}

\begin{proof}  
Let $f$ be as above.   We shall show that there is a graded affine map $f_0$ such that
  $f_0^{-1}\circ f$ is the identity map.  
  After replacing $f$ with $L_{f(0)^{-1}}\circ f$ we may    assume  $f(0)=0$. 
   We identify  $N$ with  $\hat N$ and view $f$ as a map $\mathcal F\supset \hat N\ra \hat N\subset \mathcal F$.

  We first consider the case $F=\mathbb R$. 
 By  Theorem \ref{thm_main},    $f$ is the restriction to $\hat N$  of  a map  of the form 
 $\psi^{\eps}\circ g$ where $g\in \gl(n,\R)$, $\eps\in \{0,1\}$. 
 Recall that  the automorphism $\tau=I_{\Pi}\circ \Phi_0$ of $\text{GL}(n, \mathbb R)$ 
  induces a graded automorphism (again denoted by $\tau$) of $N=\hat N$ and acts on $\mathcal F$ as $\Pi\circ \psi$.   By  replacing $f$ with $\tau\circ f$ if necessary (when $\epsilon=1$),  
  we may assume $\eps=0$  and so $f=g|_{\hat N}$.   
    This implies $g\in P^-$ as the stabilizer of $0$ is $P^-$.   So $g=(g_{ij})$ 
  is a lower triangular matrix. 
   Now we can further assume that the entries on the diagonal of $g$ are $1$, after replacing $g$ with $D^{-1}g$, where $D=\text{diag}(g_{11}, \cdots, g_{nn})$.
  So now $g$ is a lower triangular matrix with $1$s on the diagonal and such that 
   $g(\hat N)=\hat N$.  We next show that $g=I_n$.

    Suppose $g\not=I_n$.    We shall find a flag $F\in \hat N$ such that $g(F)\notin \hat N$,
     contradicting  $g(\hat N)=\hat N$.  Let $1\le k\le n-1$  be the integer such that 
     $g_{ij}=0$ for all  $i>j>k$  and $g_{jk}\not=0$ for some  $j>k$.     Let $j_0>k$ be such that  $g_{j_0k}\not=0$ and $g_{jk}=0$ for all $j>j_0$.    Denote
      $v_{j_0}=e_k-\sum_{j=k+1}^{j_0}  g_{jk} e_j$.
     Let $F=\{W_j\}$ be the flag defined by   $W_j=W^-_j$ for $1\le j\le n-j_0$ and $n-k<j\le n$, $W_{n-j_0+1}=\text{span}\{e_n, \cdots, e_{j_0+1}, v_{j_0}\}$ and $W_j=\text{span}\{e_n, \cdots, e_{j_0+1}, v_{j_0}, e_{j_0-1}, \cdots, e_{n-j+1}\}$ for $n-j_0+2\le j\le n-k$.
     It is straightforward to check that $F\in \hat N$.  However, 
         $g(v_{j_0})=e_k$  
           and 
        $g(W_{n-j_0+1})\cap W^+_{j_0-1}$ contains the nontrivial element  $e_k$ and so  $g(F)\notin \hat N$.   
    
The proof in the complex and quaternion cases are the same except we use Theorem \ref{thm_mainF} and in the complex case we may need to compose $f$ with the complex conjugation.

\end{proof}

\bibliography{product_quotient}
\bibliographystyle{amsalpha}

\end{document}

%% file: macros_kmx.tex
\usepackage{amsmath,amssymb,amsthm,graphicx,color,amscd}
\usepackage[usenames,dvipsnames]{xcolor}
\usepackage{enumitem}

\usepackage[small,nohug,heads=vee]{diagrams}

\usepackage{hyperref}

\numberwithin{equation}{section}
\hypersetup{bookmarksdepth=3}

\theoremstyle{plain}
\newtheorem{theorem}[equation]{Theorem}

\newtheorem{corollary}[equation]{Corollary}
\newtheorem*{claim}{Claim}
\newtheorem{proposition}[equation]{Proposition}
\newtheorem{lemma}[equation]{Lemma}

\theoremstyle{definition}
\newtheorem{definition}[equation]{Definition}

\theoremstyle{remark}

\newcommand{\acts}{\operatorname{\curvearrowright}}
\newcommand{\ad}{\operatorname{ad}}
\newcommand{\Ad}{\operatorname{Ad}}

\newcommand{\al}{\alpha}

\newcommand{\aut}{\operatorname{Aut}}

\newcommand{\be}{\beta}
\newcommand{\ben}{\begin{enumerate}}
\newcommand{\bit}{\begin{itemize}}
\newcommand{\C}{\mathbb{C}}

\newcommand{\cent}{\operatorname{Center}}

\newcommand{\D}{\partial}
\newcommand{\de}{\delta}

\newcommand{\diag}{\operatorname{diag}}

\newcommand{\een}{\end{enumerate}}
\newcommand{\eit}{\end{itemize}}

\newcommand{\eps}{\varepsilon}

\newcommand{\fg}{\mathfrak{g}}
\newcommand{\fh}{\mathfrak{h}}
\newcommand{\fk}{\mathfrak{k}}
\newcommand{\fl}{\mathfrak{l}}
\newcommand{\fn}{\mathfrak{n}}

\newcommand{\fp}{\mathfrak{p}}

\newcommand{\F}{\mathbb{F}}
\newcommand{\f}{\mathcal{F}}
\newcommand{\ga}{\gamma}

\newcommand{\gl}{\operatorname{GL}}
\newcommand{\gll}{\operatorname{\fg\fl}}
\newcommand{\h}{\mathcal{H}}
\renewcommand{\H}{\mathbb{H}}

\newcommand{\id}{\operatorname{id}}

\newcommand{\la}{\lambda}
\newcommand{\La}{\Lambda}

\newcommand{\loc}{\operatorname{loc}}
\newcommand{\lra}{\longrightarrow}

\newcommand{\om}{\omega}
\newcommand{\Om}{\Omega}

\newcommand{\pgl}{\operatorname{PGL}}

\newcommand{\R}{\mathbb{R}}
\newcommand{\ra}{\rightarrow}

\newcommand{\rank}{\operatorname{rank}}
\definecolor{gray}{gray}{0.7}

\newcommand{\restr}{\mbox{\Large \(|\)\normalsize}}

\newcommand{\si}{\sigma}
\newcommand{\Si}{\Sigma}
\renewcommand{\sl}{\operatorname{SL}}

\newcommand{\Span}{\operatorname{span}}

\newcommand{\stab}{\operatorname{Stab}}

\renewcommand{\th}{\theta}

\newcommand{\vol}{\operatorname{vol}}

\newcommand{\we}{\wedge}

\newcommand{\wt}{\operatorname{wt}}

\def\XXint#1#2#3{{\setbox0=\hbox{$#1{#2#3}{\int}$ }
\vcenter{\hbox{$#2#3$ }}\kern-.6\wd0}}
